\documentclass[11pt,twoside,fullpage]{amsart}
\title{A ratio of integration between
  quotients \mbox{in geometric invariant theory}}
\author{Zachary Maddock}
\thanks{This work was supported by the NSF through a Graduate Research
  Fellowship.} 
\date{September 7, 2013}
\usepackage{latexsym}
\usepackage{amssymb}
\usepackage{amsfonts}
\usepackage{amstext}
\usepackage[initials]{amsrefs}
\usepackage{amsmath}
\usepackage{times}
\usepackage{graphicx}
\usepackage{enumerate}
\usepackage{url}
\usepackage[hypertexnames=false]{hyperref}

\DeclareFontFamily{OMS}{rsfs}{\skewchar\font'60}
\DeclareFontShape{OMS}{rsfs}{m}{n}{<-5>rsfs5 <5-7>rsfs7 <7->rsfs10 }{}
\DeclareSymbolFont{rsfs}{OMS}{rsfs}{m}{n}
\DeclareSymbolFontAlphabet{\scr}{rsfs}

 \renewcommand{\theequation}{\arabic{section}.\arabic{subsection}.\arabic{equation}}

\renewcommand{\thesubsection}{\arabic{section}.\arabic{subsection}\setcounter{equation}{0}}

\newcommand{\Z}{\mathbb{Z}}

\newcommand{\C}{\mathbb{C}}
\newcommand{\Q}{\mathbb{Q}}
\newcommand{\G}{\mathbb{G}}
\newcommand{\N}{\mathbb{N}}
\newcommand{\A}{\mathbb{A}}
\newcommand{\F}{\mathbb{F}}
\renewcommand{\P}{\mathbb{P}}

\newcommand{\inv}{^{-1}}
\newcommand{\inject}{\hookrightarrow}

\newcommand{\surject}{\twoheadrightarrow}

\newcommand{\OO}{\mbox{$\mathcal{O}$}}

\newcommand{\sE}{\scr{E}}

\newcommand{\sL}{\scr{L}}
\newcommand{\sM}{\scr{M}}

\DeclareMathOperator{\stab}{stab}
\DeclareMathOperator{\red}{red}

\DeclareMathOperator{\Spec}{Spec}

\DeclareMathOperator{\id}{id}
\DeclareMathOperator{\Char}{char}

\newcommand{\dd}{/\mspace{-6.0mu}/}
\newcommand{\Cdot}{\raisebox{-0.1ex}{\scalebox{1.5}{$\cdot$}}}

\newcounter{repeatcounter}

\newtheorem{repeatamp}{Amplification}

\newtheorem{theorem}[equation]{Theorem}
\newtheorem{amplification}[equation]{Amplification}
\newtheorem{lemma}[equation]{Lemma}
\newtheorem{proposition}[equation]{Proposition}

\newtheorem{corollary}[equation]{Corollary}
\newtheorem{remark}[equation]{Remark}

\newtheorem{question}[equation]{Question}

\theoremstyle{definition}
\newtheorem{definition}[equation]{Definition}

\input xy 
\xyoption{all}

\newcommand{\ctop}{c_{\romantop}}
\DeclareMathOperator{\romantop}{top}

\newcommand{\rootctop}{\Delta}
\newcommand{\arbclass}{c}

\newcommand{\chargp}[1]{ \Lambda^\ast({#1})}
\newcommand{\cochargp}[1]{ \Lambda_\ast({#1})}
\newcommand{\gitstack}[2] { [{#1}^{ss}_{#2}/{#2}]}

\newcommand{\semistable}[2]{ {#1}^{ss}_{#2}}
\newcommand{\strictlysemistable}[2]{ {#1}^{sss}_{#2}}
\newcommand{\stable}[2]{ {#1}^{ s}_{#2}}
\newcommand{\unstable}[2]{ {#1}^{un}_{#2}}
\newcommand{\comment}[1]{}

\newcommand{\alphanew}{\sigma}
\newcommand{\betanew}{\tau}

\DeclareMathOperator{\rank}{rank}
\DeclareMathOperator{\Sym}{Sym}

\begin{document}
\maketitle

\begin{abstract}
Let $T$ be a maximal torus of a connected reductive group $G$ that acts
  linearly on a projective variety $X$ so that all semi-stable points
  are stable.
  This paper compares the integration on the geometric invariant
  theory quotient $X \dd G$ of Chow classes $\sigma$ 
  to the integration 
  on the geometric invariant theory quotient $X \dd T$ of certain lifts
  of $\sigma$  
  twisted by  $\ctop(\mathfrak g / \mathfrak t)$,
  the top Chern class of the $T$-equivariant vector bundle
  induced by the quotient of the adjoint representation on the
  Lie algebra of $G$ by that of $T$. We provide a purely algebraic
  proof that the ratio between any two such integrals is an
  invariant of  the group $G$ and that
  it equals the order of the Weyl group whenever
  the root system of $G$ decomposes into irreducible
  root systems of type $\mathbf {A}_n$, for various $n \in \N$.
  As a corollary, we are able to remove this restriction on root
  systems by applying a related result of Martin from symplectic geometry. 
\end{abstract}

\section*{Introduction}

\setcounter{equation}{0}

\renewcommand{\theequation}{0.\arabic{equation}}
\renewcommand{\thesubsection}{0.\arabic{subsection}\setcounter{equation}{0}} 

\subsection*{A brief history}
The cohomology of quotients arising from
geometric invariant theory (GIT)  
has been the object of extensive study.
In 1984, Kirwan  \cite{kir1} integrated
the previous works of Hesselink,
 Kempf and Ness \cites{hes1, hes2, kem1, kem-nes1} to
explore the
structure of GIT quotients from both the algebraic and symplectic
perspectives, ultimately finding formulas to compute
 Hodge numbers.  
Five years later, Ellingsrud and Str\o mme \cite{ell-str1}  began to
study 
the relationship
between the Chow rings of the two GIT quotients
 $\P^n_{\bar k} \dd G$ and $\P^n_{\bar k}\dd T$, for
a reductive group $G$ over an algebraically closed field $\bar k$ with
maximal torus $T \subseteq G$ acting on $\P^n_{\bar k}$ so
that all semi-stable points have trivial stabilizers;  
their main result was a presentation of the Chow ring
$A^\ast(\P^n \dd G)_{\Q}$ in terms of explicit generators and
relations.
Brion \cite{bri2} then expanded
this relationship to arbitrary linear actions of connected reductive
groups $G$ on smooth, projective varieties $X$ over the complex
numbers, proving that the $G$-equivariant cohomology of
the locus of $G$-semi-stable points is
isomorphic to the subgroup of 
Weyl anti-invariant classes of the $T$-equivariant cohomology 
group of the larger locus
of $T$-semi-stable points:
\begin{equation*}
\phi:  H^\ast_G(\semistable X G;\Q) \stackrel {\cong}{\to} H^\ast_T (\semistable X T;\Q) ^a.
\end{equation*}
Later 
 Brion and Joshua \cite{bri-jos1} extended these results further to the case of
singular $X$, but with equivariant \emph{intersection} cohomology used as a suitable
replacement for the standard theory.  

Brion's construction of the isomorphism $\phi$ is as follows
(see~\cite{bri2} for full details).
As
$\semistable X G$ is a $G$-variety, the inclusion $T \subseteq G$
induces a homomorphism
$\pi^\ast: H_G^\ast(\semistable X G;\Q) \to H_T^\ast(\semistable X
G;\Q)$. Because $T$ is a maximal torus, $\pi^\ast$ 
yields an 
isomorphism onto the submodule of elements invariant under the action
of the Weyl group $W$:
$$\pi: H_G^\ast(\semistable X G;\Q) \stackrel{\cong}{\to}
H_T^\ast(\semistable X G;\Q)^W.$$
Moreover, there is a
$W$-equivariant isomorphism 
\begin{equation*}
H_T^\ast(\semistable X G; \Q) \cong  S \otimes_{S^W}
H_G^\ast(\semistable X G;\Q),
\end{equation*}
where $S := H_T^\ast(\Spec \C;\Q)$ is the
$T$-equivariant cohomology of the 
point $\Spec \C$, and under this identification $\pi^\ast$ becomes $1 \otimes
\id$. 
The $W$-anti-invariant elements $S^a \subseteq S$ form a free
module of rank~$1$
over the subring $S^W$ of Weyl-invariant elements, and a generator is
given by
$$\rootctop :=
\ctop(\mathfrak g/ \mathfrak b),$$
the top equivariant Chern class of the 
adjoint representation on $\mathfrak g / \mathfrak b$, where
$\mathfrak g$ is the Lie algebra of $G$ and $\mathfrak b$ is the Lie
algebra of a Borel subgroup containing $T$.
Therefore,  $\rootctop \smile \pi^\ast(-)$ gives an
isomorphism from $H_G^\ast(\semistable X G; \Q)$ onto the submodule of
$W$-anti-invariant elements of $H_T^\ast(\semistable X G; \Q)$,
$$\rootctop \smile \pi^\ast(-): H_G^\ast(\semistable X G; \Q)
\stackrel {\cong}\to
H_T^\ast(\semistable X G;\Q)^a.$$
 The open inclusion $i: \semistable X G \inject \semistable X
T$ induces a 
 homomorphism  $i^\ast: H_T^\ast(\semistable X T; \Q) \to
H_T^\ast(\semistable X G; \Q)$, and Brion's key observation is that
$i^\ast$  is an isomorphism on the $W$-anti-invariant
submodules:
$$i^\ast:  H_T^\ast(\semistable X T; \Q)^a \stackrel \cong \to
H_T^\ast(\semistable X G; \Q)^a.$$
The composition 
$\phi:= (i^\ast)\inv \circ ( \rootctop \smile \pi^\ast)$ yields
the desired isomorphism.  Explicitly,  if $\tilde \alphanew \in H_T^\ast(\semistable X
T; \Q)^W$ denotes some $W$-invariant lift of the class 
$\alphanew \in H_G^\ast(\semistable X G; \Q)$, that is if $i^\ast \tilde
\alphanew = \pi^\ast \alphanew$, then $\phi$ can be described
as 
$$\phi: \alphanew \mapsto \rootctop \smile \tilde \alphanew.$$

\subsection*{The main goal}
This paper addresses the question of how
the isomorphism $\phi$ interacts with the integration pairings on 
the GIT quotients $X\dd
G$ and $X \dd T$.
When $\semistable X G = \stable X G$, there is a natural identification
between the equivariant cohomology groups of the semi-stable locus and
the ordinary cohomology groups 
of the GIT quotient,
$$H_G^\ast(\semistable X G;\Q) \cong H^\ast(X\dd G;\Q),$$
(and similarly with $T$ in place of $G$).
For any $\alphanew_1, \alphanew_2 \in H^\ast(X\dd G; \Q)$,
one can then, in such a case, compare the integrals 
$$\int_{X\dd G} \alphanew_1 \smile \alphanew_2~~ \stackrel ? \leftrightarrow
~~\int_{X\dd T} \phi(\alphanew_1) \smile \phi(\alphanew_2).$$
Because $\phi(\alphanew_1) 
\smile \phi(\alphanew_2) = (\rootctop \smile \rootctop)\smile
({\tilde\alphanew_1 \smile \tilde \alphanew_2})$ and $i^\ast(\tilde \alphanew_1 \smile
\tilde \alphanew_2) =\pi^\ast(\alphanew_1 \smile
  \alphanew_2)$, we may 
simplify the expression by defining $\alphanew := \alphanew_1
\smile \alphanew_2$.  Moreover, 
we consider the class $\ctop(\mathfrak g/ \mathfrak
t)$, where $\mathfrak t \subseteq \mathfrak g$ is the inclusion of
the Lie algebra of $T$ in the adjoint representation on the Lie
algebra of $G$, 
instead of the class $\rootctop \smile \rootctop$, which
just differs from the former by the sign $(-1)^{\dim \mathfrak g /\mathfrak
b}$.   
After these substitutions, the question becomes the comparison of the
integrals $\int_{X\dd G} \alphanew$ and $\int_{X \dd T} \ctop(\mathfrak g / \mathfrak t) \smile
\tilde \alphanew$ for $\alphanew \in H^{\ast}(X\dd G;\Q)$. 

Within an unpublished manuscript, 
Martin \cite{mar1} provided an answer to the 
symplecto-geometric analogue of this question.
There he
proved the following formula
for Hamiltonian actions of connected compact Lie groups $G$ on
symplectic manifolds $X$ for which the moment map is proper and has
$0$ as a regular value:
\begin{equation}
 \int_{X\dd G}\alphanew = \frac 1 {|W|} \int_{X\dd T}  \ctop(\mathfrak g / \mathfrak t)
\smile \tilde \alphanew, \label{martins-integration-formula}
\end{equation}
with $X\dd G$ and $X\dd T$ here denoting the symplectic reductions.  
Technically \eqref{martins-integration-formula} only applies when $G$
and $T$ act freely on $\mu_G\inv(0)$ and $\mu_T\inv(0)$ respectively, but
Martin proves a version of the formula that
applies when the stabilizer groups of these actions are finite; this
version involves a multiplicative factor equal to the ratio of the
orders of the generic stabilizers (cf.~\cite[Thm.~B$'$]{mar1}).   

Martin's formula 
may be deduced
from the Jeffrey-Kirwan-Witten non-abelian localization theorem
 \cites{JK,wit}, although the proof Martin gave is much simpler than the
proof of the general theorem \cite[Thm. 8.1]{JK}.
In addition to Martin's work \cite{mar1}, there has been a large body
of literature devoted to understanding non-abelian localization;
alternative approaches to the theorem may be found in the works of
Guillemin and Kalkman \cite{GK}, Paradan \cite{par}, and Vergne
 \cite{ver}.
We note that the methods used in these works are 
analytic, and hence bound to characteristic $0$,
while the
methods used in the works of Brion, Ellingsrud-Str\o mme, and
Brion-Joshua referenced above are algebraic.

\subsection*{New results}
This article generalizes Martin's result
to the algebraic setting of varieties $X$  
over an arbitrary field $k$.  Let $G$ be a reductive group over
$k$ with 
a maximal torus $T \subseteq G$, and 
let $X$ be a projective 
(and possibly singular) $G$-variety with a $G$-linearized
ample line bundle $\sL$ for which $\stable X T =
\semistable X T \neq \emptyset$.
For any Chow $0$-cycle $\alphanew \in A_0(X\dd
\G)_{\Q}$ with non-zero degree, 
we are led to define
the \emph{GIT integration ratio},
$$r_{G,T}^{X,\alphanew} := \frac{\int_{X\dd T} \ctop(\mathfrak g / \mathfrak t) \frown \tilde \alphanew}
{\int_{X\dd G} \alphanew} \in \Q,$$
where the Chow class $\tilde
\alphanew \in A_\ast(X\dd T)_{\Q}$ denotes an arbitrary lift of the class
$\alphanew$ ~(see Defn.~\ref{definition-of-lift}).

Understanding the 
 properties of the GIT integration ratio
will guide us to the correct generalization of Martin's theorem.
Since Definition \ref{definition-of-lift} involves the pull-back
of Chow classes by the morphism of Deligne-Mumford stacks
$[\semistable X G/T] \to [\semistable X G / G]$, information
about the stabilizer groups of the $G$ and $T$ actions is already
encoded in our construction, and we shall not require
a separate statement in the presence of non-trivial stabilizers.

\begin{remark}
We may immediately reduce our
discussion to the case of a connected reductive group $G$
because $\semistable X G = \semistable X {G_0}$,
where $G_0 \subseteq G$ is the
connected component of the identity
(cf.~\cite[Prop.~1.15]{GIT}), from which it follows quickly that 
\begin{equation*}
r_{G,T}^{X, \alphanew} = [G:G_0]\cdot r_{G_0,T}^{X,
  \pi^\ast\alphanew},
\end{equation*}
since the morphism $\pi: [\semistable X {G_0}/ G_0] \to [\semistable
  X G/G]$ 
induced by the
quotient $\semistable X {G_0}  \to
[\semistable X G/G]$
is finite and flat of degree $[G:G_0]$.  
\end{remark}

From now on, we assume
that $G$ is a connected reductive group and $W$ is the associated Weyl group.
Our first main result 
(proved in   
   \newcounter{sectionSafety}
   \setcounter{sectionSafety}{\value{equation}}
   \S \ref{section-independence-of-GIT-ratio})
   \setcounter{equation}{\value{sectionSafety}}
concerns the invariance properties of the
GIT integration ratio $r_{G,T}^{X,\alphanew}$, which \emph{a priori}
could depend on the choice of 
lift $\tilde \alphanew$ or the variety $X$.  In fact, neither is the case:

\begin{theorem}\label{GIT-integral-ratio-is-invariant-of-G-theorem}
  Let $G$ be a connected reductive group over an arbitrary field and $T \subseteq G$ a 
  maximal torus.  If $X$ is a
  projective $G$-variety with a
  $G$-linearized ample
  line bundle $\sL$ for which $\semistable X T = \stable X T$ and
  $\alphanew \in A_0(X\dd G)$ is a Chow $0$-cycle
  satisfying
  $\int_{X\dd G} \alphanew \neq 0$, then
  the  GIT integration ratio $r_{G,T}^{X,\alphanew}$ defined as above
  depends not on the choice of $\sigma$, $T$, $\sL$, or $X$.
  That is, $r_G := r_{G,T}^{X,\alphanew}$ is an invariant of the group
  $G$.
\end{theorem}

Connected reductive groups decompose, up to central isogeny, as the direct
product of a torus and simple algebraic groups
(cf.~\cite[\S14]{bor1}), and 
the GIT integration ratio can be computed in terms of the GIT
integration ratios of the simple algebraic groups appearing in this
decomposition. This is
the content of our second main result (proved in  
\S\setcounter{sectionSafety}{\value{equation}}\ref{section-functorial-properties}): 
\setcounter{equation}{\value{sectionSafety}}

\begin{theorem}\label{GIT-integral-ratio-decomposes-multiplicatively}
  If $G$ is a connected reductive group over a field $k$ and 
  there exists a central 
  isogeny, 
  $G_1 \times \cdots \times G_n \times S \surject G,$
  for connected reductive groups $G_i$ and a torus $S$, then 
  $$r_G = \prod_{i =1}^n r_{G_i}.$$
\end{theorem}

As a result of these theorems, the determination of the
value $r_G$ for connected reductive groups $G$ is reduced to the
computation of $r_G$ on
a single example for each simple group.  We
do this explicitly in
\S\setcounter{sectionSafety}{\value{equation}}\ref{section-calculation}\setcounter{equation}{\value{sectionSafety}}
for the simple group 
$G= PGL(n)$, where we verify $r_G = n! = |W|$.  Thus, we obtain a strictly
algebraic proof of the following corollary:

\begin{corollary}\label{main-theorem}
  Let $G$ be a connected reductive group over a field $k$ and $T
  \subseteq G$ a  
  maximal torus.  If the 
  root system of $G$ decomposes into irreducible root systems of type
  $\mathbf A_{n}$, for various $n \in \N$, then for any
  $G$-linearized ample line bundle on a
  projective $G$-variety $X$ over $k$ satisfying $\stable X T = \semistable X
  T$ and any
  Chow class $\alphanew \in A_0(X \dd G)_\Q$
  with lift $\tilde \alphanew \in A_{\ast}(X\dd T)_\Q$,
  \begin{equation*}
    \int_{X\dd G}\alphanew = \frac{1}{|W|} \int_{X\dd T} \ctop(\mathfrak g / \mathfrak t) \frown
    \tilde \alphanew.
  \end{equation*}
\end{corollary}

We can do even better than  Corollary \ref{main-theorem} if we do not
restrict ourselves to purely algebraic arguments.  
Indeed, we may remove the restriction on root systems,
proving that
the GIT integration
ratio is equal to the order of the Weyl group
for any connected reductive group:

\begin{amplification}\label{amplification-general-result}
  Let $G$ be a connected reductive group over a field $k$ and $T
  \subseteq G$ a  
  maximal torus.  
  For any
  $G$-linearized ample line bundle on a
  projective $G$-variety $X$ over $k$ satisfying $\stable X T = \semistable X
  T$ and any
  Chow class $\alphanew \in A_0(X \dd G)_\Q$
  with lift $\tilde \alphanew \in A_{\ast}(X\dd T)_\Q$,
  \begin{equation*}
    \int_{X\dd G}\alphanew = \frac{1}{|W|} \int_{X\dd T}
    \ctop(\mathfrak g / \mathfrak t) \frown  \tilde \alphanew.
  \end{equation*}
\end{amplification}
The proof of this result
uses the theories of relative GIT (cf.~\cite{sesh1}) and specialization
(cf.~\cite[\S20.3]{ful1}), along with Theorem
\ref{GIT-integral-ratio-is-invariant-of-G-theorem},  to
reduce the proof 
to the case of connected reductive groups over
the complex numbers, where  Martin's result  \cite[Thm.~B$'$]{mar1}
applies.  
This argument is found near the end of the paper
in \S \ref{section-final-remarks}.

\vspace{.2cm}
\noindent {\bf \large Acknowledgments.} 
 It is with pleasure that I thank my thesis advisor Johan de Jong, both for
 teaching me
 algebraic geometry and for his generous guidance that led me through 
 the discovery of these results.  
 I also thank Burt Totaro and the anonymous referees
 for their diligence in reading earlier manuscripts and for their 
 valuable suggestions.

\section*{Notation}
\setcounter{equation}{0}
Reductive groups
\begin{itemize}
\renewcommand{\labelitemi}{\Cdot}
\item $k$ denotes a field and $\bar k$ an algebraic closure of $k$.
\item $e \in T \subseteq B \subseteq G$ denotes a smooth, connected
  reductive group $G$ over $k$ with  identity 
  element $e$, a  maximal torus $T$, and  a Borel subgroup $B$.
\item $\mathfrak g$, $\mathfrak b$, and $\mathfrak t$ denote
  respectively the Lie algebras of $G$, $B$,
  and $T$.
\item   $N_T \subseteq G$ denotes the normalizer of $T$ in $G$.
\item $W := T \backslash N_T$  denotes the Weyl group of
 right cosets of $T$ in its normalizer $N_T$.
\item $\Phi$ denotes the root system corresponding to
  $G_{\bar k}$ and the maximal
  torus $T_{\bar k}$.
\item $\chargp T$ denotes the character group of $T$ and $\cochargp T$
  the group of $1$-parameter subgroups.
\item $\bar G$ denotes a group that is the quotient of $G$ by some
  normal subgroup.
\end{itemize}

Group actions and quotients
\begin{itemize}
\renewcommand{\labelitemi}{\Cdot}
\item $V$ denotes a finite dimensional $G$-representation over $k$.
\item $\P(V)$ denotes the projective space of hyperplanes in $V$, so that $\Gamma(\P(V), \OO(1)) = V$.
\item $X$ denotes a projective variety over $k$ with a \emph{right} $G$-action
  and an ample $G$-linearized line bundle $\sL$.
\item $\stable X G$, $\semistable X G$, $\strictlysemistable X G$ and $\unstable X G$ denote the loci of
  stable, semi-stable, strictly semi-stable, and unstable points for
  a specified $G$-linearized ample line bundle on a projective $G$-variety $X$.
\item $X \dd G$ the uniform categorical
  quotient of the semi-stable locus $\semistable X G$ by right $G$-action.
\item $Y \times_H G$ denotes, for an algebraic space $Y$ with a right
  $G$-action and a closed subgroup $H \subseteq G$,  the
  algebraic space comprising the
  quotient of the product
  $Y \times G$ by the set-theoretically free $H$-action $(y, g)\cdot h
  := (y\cdot h, h\inv \cdot g)$.  If $Y$ is a scheme, then $Y \times_H
  G$ is often a scheme as well (cf.~\cite[Prop.~23]{edi-gra1}).
\item  $[Y/G]$ denotes the stack-theoretic quotient of a right $G$-variety $Y$ by the $G$-action.
\item $BT := [\Spec k/T]$ denotes the Artin stack that is the
  algebraic classifying space of $T$.
\end{itemize}

Chow theory
\begin{itemize}
\renewcommand{\labelitemi}{\Cdot}
\item $A_\ast(-)_R$ (resp.~$A^\ast(-)_R$) denotes the Chow group
  (resp.~operational Chow group), graded by dimension
  (resp.~codimension),  and with coefficients
  in $R$, an abelian group.  Coefficients will be
  taken in $\Z$ by default when the group $R$ is omitted from notation. 
\item $A^G_\ast(-)_R$ (resp.~$A_G^\ast(-)_R$) denotes the $G$-equivariant
  Chow group (resp.~$G$-equivariant operational Chow group) with
  coefficients in an abelian group $R$.
\item $\int_Y \alphanew \in R$ denotes the degree of a Chow class $\alphanew \in
  A_0(Y)_R$ on a proper variety $Y$ over $k$, computed via proper
  push-forward by the structure morphism.
\item $\rootctop := \ctop(\mathfrak g/\mathfrak b) \in A^\ast(BT)$ and
  $\ctop(\mathfrak g/\mathfrak t) \in A^\ast(BT)$ are the
   top  Chern classes of the universal $T$-equivariant vector bundles induced by
  the adjoint representations.
\end{itemize}

\renewcommand{\theequation}{\arabic{section}.\arabic{subsection}.\arabic{equation}}

\renewcommand{\thesubsection}{\arabic{section}.\arabic{subsection}\setcounter{equation}{0}}

\section{Lifting classes on smooth varieties}\label{section-rootctop-on-smooth-x}
\setcounter{equation}{0}

The goal of this section is to prove if $X$ is a smooth
$G$-variety
over an arbitrary field $k$ with a $G$-linearized ample line bundle
for which $\semistable X T = \stable X T$, 
then any Chow class $\alphanew \in A_\ast(X \dd G)_\Q$ gives rise to a 
well-defined class $\rootctop \frown \tilde \alphanew \in A_\ast(X \dd
T)_\Q$ that is independent of the choice of lift $\tilde \alphanew$,
as defined in \S\ref{subsection-definition-of-lift}.
As a consequence, the
independence of the Chow class $\ctop(\mathfrak g / \mathfrak t)
\frown \tilde \alphanew$ follows immediately since
$$\ctop(\mathfrak g /\mathfrak t) = (-1)^{|\Phi|/2} \rootctop
\frown \rootctop.$$
 It is this independence that is needed to make the GIT integration
 ratio $r_{G,T}^{X, \alphanew}$ well-defined.
The case of singular $X$ is treated separately in
\S\ref{strictly-semi-stable-points-section}.

\subsection{Geometric invariant theory}
We begin by briefly reviewing geometric invariant
theory, mainly to set our conventions.  

\begin{definition}
Let $X$ be a 
projective
$G$-variety over $k$ and $\sL$ a
$G$-linearized ample line bundle on $X$, that is, an ample
line bundle $\pi:\sL \to X$ along with $G$-actions on $\sL$ and $X$
for which $\pi$ is
equivariant and on whose fibres $G$ acts linearly.   
\begin{itemize}
\renewcommand{\labelitemi}{\Cdot}
\item  The  \emph{semi-stable locus} is the open subscheme defined by
  $$\semistable X G := \{x \in X: \exists ~n > 0 \textrm{
    and some } \phi \in \Gamma(X, \sL^{\otimes n})^G \textrm{
  satisfying } \phi(x) \neq  0\}.$$
\item  The \emph{stable}\footnote{Sometimes in the literature this is
    referred to as ``properly stable''.} \emph{locus} is the open
  subscheme defined by 
  $$\stable X G := \{ x \in \semistable X G:  x\cdot G \subseteq
  \semistable X G \textrm{ is a 
    closed subscheme and } |\stab_G x| < \infty \}.$$
\item The \emph{unstable locus} is defined to be
  $\unstable X G := X \setminus \semistable X G.$
\item The \emph{strictly semi-stable locus} is defined to be
  $\strictlysemistable X G := \semistable X G \setminus \stable X G.$
\end{itemize}
\end{definition}
\noindent The following theorem justifies the making of the above
definitions:

\begin{theorem}[Mumford]\label{mumford-main-GIT-theorem}
  The semi-stable locus
  for a $G$-linearized line bundle $\sL$ on a projective $G$-variety $X$ 
  admits a uniform categorical quotient $\pi: \semistable X G \to X\dd
  G$ called the \emph{GIT quotient of $X$ by $G$.}
  Moreover, some positive tensor power of $\sL$
  descends to an ample line bundle  
  on the projective variety $X\dd G$, 
  and
  the restriction
  of $\pi$ to the stable locus is a geometric quotient.
\end{theorem}
\begin{proof}
  See  \cite[Thm.~1.10]{GIT}.  
\end{proof}

Next we describe how to compute the stable and
semi-stable loci in practice.  As these loci are compatible with field
extensions (cf.~\cite[Prop. 1.14]{GIT}), we may assume that we are
working over the algebraically closed field $\bar k$.
Let $T\subseteq G$ be a maximal torus with
character group $\chargp{T}$. Equivariantly embed $X$
into $\P(V)$ for some $G$-representation $V$ so that $\OO_{\P(V)}(1)$
pulls-back to the induced $G$-linearized line
bundle $\sL^{\otimes n}$, for some $n > 0$. 
We will call such an embedding \emph{compatible}.
The $G$-representation structure on $V$ endows
$\P(V)$ with a $G$-action
and a $G$-linearization on $\OO(1)$
for which both $\semistable{ \P(V)} G \cap X =
\semistable X G$ and $\stable {\P(V)} G \cap X = \stable X G$.
Since $\bar k$ is algebraically closed, $T$ is diagonalizable and
hence $V$ decomposes
as the direct sum of weight spaces  $V = \oplus_{\chi
  \in \chargp T} V_\chi$.  
\begin{definition}\label{definition-of-state}
  For any $x \in X \subseteq \P(V)$ as above,  the
  \emph{state of $x$} is defined to be  
  $$\Xi(x) := \{ \chi \in \chargp T: \exists~ v \in V_\chi
  \textrm{ such that } v(x) \neq 0 \}.$$
\end{definition}
We now state in the above notation the following well-known numerical
criterion for stability:
\begin{theorem}[Hilbert-Mumford criterion]
  \label{hilbert-mumford-criterion}
  A point $x \in X$ is $T$-semi-stable for the induced $T$-linearization
  on $\sL$
  if and only if $0$ is in the convex hull of $\Xi(x)$ in $\chargp T
  \otimes \Q$.  A point $x \in X$
  is $T$-stable if and only if $0$ is in the interior
  of the convex hull of $\Xi(x)$.  Furthermore, 
  $$\semistable X G = \bigcap_{g \in G}  \semistable X T \cdot g,~~
  \textrm{ and }~~  \stable X G = \bigcap_{g \in G}  \stable X T \cdot
  g.$$
\end{theorem}
\begin{proof}
  See \cite[Thm.~2.1]{GIT}.
\end{proof}

\begin{remark}\label{remark-T-stability-and-G-stability}
  A corollary of the Hilbert-Mumford criterion is that the condition
  in $T$-stability that all semi-stable points are stable,
  $\semistable X T = \stable X T$, automatically implies the analogous
  condition in $G$-stability,  $\semistable X G = \stable X G$.
\end{remark}

We conclude this review by presenting a stratification that describes
the structure of the unstable locus.  The stratification is due to
Kirwan, but relies on  
the previous works of Hesselink \cites{hes1, hes2}, Kempf \cite{kem1} and
Ness \cite{kem-nes1}.

\begin{theorem}[Kirwan]\label{theorem-kirwan-strata-decomposition}
  Let $X$ be a projective $G$-variety, along with a $G$-linearized
  ample line bundle $\sL$, over an algebraically 
  closed field $\bar k$. 
  The unstable locus $\unstable X G$ admits a
  $G$-equivariant stratification,
  $$\unstable X G = \bigcup_{\beta \in \mathbf B} S_\beta,$$ 
  indexed by  a finite partially ordered set  $\mathbf B$, with
  the following properties: 
\begin{enumerate}
\item $S_\beta \subseteq \unstable X G$ is a locally closed
  $G$-equivariant subscheme. 
\item $S_\beta \cap S_{\beta'} = \emptyset$ for $\beta \neq \beta'$.
\item $\overline {S_\beta} \subseteq \bigcup_{\beta' \geq \beta}
  S_{\beta '}$.
\item There exist parabolic subgroups $T \subseteq P_\beta \subseteq
  G$, locally closed
  $P_\beta$-closed subschemes $Y_\beta \subseteq S_\beta \cap
  \unstable X T$,
  and a surjective $G$-equivariant morphism,
  $$\phi: Y_\beta \times_{P_\beta} G \to S_\beta,$$
  \label{item-isomorphism-when-smooth}
  induced by the multiplication morphism $Y_\beta \times G \to S_\beta$.
\end{enumerate}
If moreover $X$ is smooth, then there exists a 
surjection of reductive groups $\pi:G \surject \bar G$ such that the
$G$-actions on $X$ and $\sL$ are induced by $\bar G$-actions,
and $\phi$ descends to an isomorphism
$\bar \phi: Y_\beta \times_{\bar P_\beta} \bar G  \to S_{\beta}$,
where $\bar P_\beta = \pi(P_\beta)$.  
If $\Char k = 0$ or $X$ compatibly embeds into $\P(V)$ for some
faithful $G$-representation $V$, then one can take $G = \bar G$.   
\end{theorem}
\begin{proof}
  If $X$ compatibly embeds into
  $\P(V)$ for some faithful $G$-representation $V$ then
  the theorem holds, with $\phi$ an isomorphism when $X$ is smooth, by 
  \cite[\S 12-\S13]{kir1}.  The general result can then be
  reduced to this case: let $\bar G$ be
  the image of the homomorphism $h:G \to GL(V)$ and let $\bar
  P_{\beta}$ be the parabolic subgroups determined by this faithful
  $\bar G$ action as in \cite[\S 12-\S13]{kir1};
  the theorem then holds with $P_{\beta} := h\inv(\bar
  P_{\beta})^{\red}$. 
  We remark that unlike $\bar \phi$, 
  the morphism $\phi$ may not be an isomorphism
  for general smooth $X$ because
  $h\inv(\bar P_{\beta})$ may be non-reduced if the characteristic
  is positive.
\end{proof}

\subsection{Lifts}
\label{subsection-definition-of-lift}

We now define precisely what it means to lift a Chow class $\alphanew
\in A_\ast(X\dd G)_\Q$ to a class $\tilde \alphanew \in A_\ast(X \dd
T)_\Q$.
We make use of the notion of Chow groups of quotient stacks,
 a review of which may be
 found in the appendix to this paper. 
The main result we will need is that  when $\semistable X G =
\stable X G$, the
quotient $\gitstack X G$ is a proper Deligne-Mumford stack with a
morphism $\phi^G: \gitstack X G \to X \dd G$ that
induces an isomorphism of rational Chow groups
(cf.~Thm. \ref{theorem-gillet-vistoli}): 
\begin{equation*}
\phi^G_\ast: A_\ast(\gitstack X G)_{\Q} \cong A_\ast(X\dd G)_\Q.
\end{equation*}
 Via the identification $\phi^G_\ast$, we may think
of a Chow class  $\alphanew \in A_\ast(\gitstack X G)_\Q$
equivalently as 
$\phi_\ast^G(\alphanew) \in A_\ast(X \dd G)_\Q$, and we will
henceforth denote both classes
by the symbol $\alphanew$.

\begin{definition}\label{definition-of-lift}
If $X$ is a projective $G$-variety with a $G$-linearized ample line bundle for which 
$\semistable X T = \stable X T$, then
a \emph{lift of} $\alphanew \in A_d(X \dd G)_\Q$ is a
class $\tilde 
\alphanew \in A_{d + \dim (G/T)}(X \dd T)_\Q$ 
that satisfies
 $i^\ast (\tilde \alphanew) =
f^\ast(\alphanew)$, where 
 $i:[\semistable X G / T] \inject \gitstack X T$
is the open immersion, and
 $f: [\semistable X G / T] \to \gitstack X G$ 
is the flat fibration with fibre $G/T$.
\end{definition}
\begin{remark}\label{remark-ctop-killing}
By the right exact sequence of Chow groups 
$$A_\ast([\unstable X G \cap\semistable X T / T])_\Q \to A_\ast(
\gitstack X T)_\Q \stackrel {i^\ast} \to A_\ast(\semistable X G / T])_\Q \to 0,$$
any two lifts of a Chow class differ
by the push-forward of an element of $A_\ast([\unstable X G \cap
\semistable X T / T])_\Q$.  
\end{remark}

\subsection{Vanishing on a stratum}\label{rootctop-is-zero-on-each-stratum-section}

Throughout this subsection, we  
assume that $X$ is a
smooth, projective $G$-variety with a $G$-linearized ample line bundle
over an algebraically closed field $\bar k$.
In light of Remark \ref{remark-ctop-killing}, to show that
$\rootctop \frown \tilde \alphanew$ is independent of the choice of lift,
it suffices to show that $\rootctop$ kills all elements in the image
of $A_\ast([\unstable X G \cap\semistable X T / T])_\Q$.  
For now we just prove
that $\rootctop$ vanishes on the $T$-semi-stable locus of each
stratum, $S_\beta \cap \semistable X T$.  Since $X$ is smooth, Theorem
\ref{theorem-kirwan-strata-decomposition} guarantees that the
stratum $S_\beta$ is fibred over a flag variety ${\bar P}_\beta \backslash
\bar G$ of right $\bar P_\beta$-cosets,
and so we begin our study here.

We require some notation related to the Weyl group $W := T \backslash N_T$. 
For a parabolic subgroup $P\subseteq G$ containing the maximal torus $T$,
denote by $W_P$ the subgroup $T \backslash (N_T \cap P) \subseteq W$.
Let the symbol $\dot w$ denote a
choice of representative in $N_T$ of a Weyl class $w \in W$, and
let the symbol $\bar w$ denote the image of $w$ in $W_P\backslash W$.

\begin{lemma}\label{calculating-localization-of-point-lemma}
  Let $P \subseteq G$ be a parabolic subgroup containing $T$, and
  denote by $\mathfrak p$ its Lie algebra.
  The inclusion $i: W_P\backslash W \to P\backslash G$, defined
  by 
  $\bar w \mapsto P \dot w $,
  is the inclusion of the $T$-fixed points of $P\backslash G$,
  and the Gysin pull-back 
  of the $T$-equivariant class $[P] \in A_\ast^T(P \backslash G)$ is given by
  $$i^\ast([P]) = 
\ctop(\mathfrak g/\mathfrak p) \cdot [ \bar e] \in A^T_\ast(W_P
\backslash W),$$ 
  where
  the $T$-action on $\mathfrak g/\mathfrak p$ is via the adjoint
  representation.
\end{lemma}
\begin{proof}
  It is well-known that the $T$-invariant points of $P \backslash G$ are
  precisely  $W_P \backslash W$.
  The Chow group $A^T_\ast(W_P \backslash W)$ is a free
  $A^\ast(BT)$-module with basis given by the elements of $W_P \backslash W$.
  The element $P \in P \backslash G$ is an isolated, nonsingular fixed
  point,
  hence $i^\ast ([P])$ equals the 
  product of $[\bar e]$ with the $T$-equivariant top Chern class of the
  normal bundle of $P$ at $P$, which is just $\ctop(\mathfrak g
  /\mathfrak p)$.
\end{proof}

\begin{lemma}\label{rootctop-is-zero-from-localization-in-G/P-lemma}
  Let $P \subseteq G$ be a parabolic subgroup containing
  $T$, 
  and let $U \subseteq P \backslash G$ denote the open complement of
  the $T$-fixed points $W_P \backslash W \inject P \backslash G$.
  As an element
  of the $T$-equivariant operational Chow group of $U$,
  $$\rootctop = 0 \in A^\ast_T(U).$$
\end{lemma}
\begin{proof}
  The variety $U$ is smooth, so by Poincar\'e duality (Theorem
  \ref{poincare-duality-theorem}) it suffices to prove that
  $\rootctop \frown [U] = 0 \in A_\ast^T(U)$.
  Let $X := P \backslash G$ denote the flag variety,
  and let $i: X^T = W_P \backslash W \to X$ denote the inclusion of the
  $T$-fixed points.  By the right-exact sequence of Chow groups
  $$A_\ast^T(X^T) \stackrel{i_\ast} \to A_\ast^T(X) \to A_\ast^T(U) \to 0,$$
  it suffices to show that $\rootctop \frown [X]$ is in the image of
  $i_\ast$. 
  Since $X$ is a smooth, projective variety, 
  by the localization theorem
  (Theorem  \ref{brions-localization-theorem}),
  $i^\ast:A^\ast_T(X) \to A^\ast_T(X^T)$ is an injective
  morphism of $A^\ast(BT)$-algebras.
  Thus, it suffices to prove that $i^\ast (\rootctop \frown [X])$ is in the
  image of 
  $i^\ast \circ i_\ast$.

  The ring $A^\ast(X^T)$ is a free $A^\ast(BT)$-module with basis
  given by $\{ [\bar w] : \bar w \in W_P \backslash W\}$.
 In terms of this basis, 
  $$i^\ast(\rootctop \frown [X]) = \sum_{\bar w \in W_P \backslash W}
  \rootctop \cdot [\bar w].$$
  By  Lemma \ref{calculating-localization-of-point-lemma}, 
  $i^\ast \circ i_\ast([\bar e]) = 
       \ctop(\mathfrak g/ \mathfrak p)\cdot  [\bar e].$
  Since $i$ is a $W$-equivariant inclusion,  $i^\ast \circ i_\ast$ is compatible with
  the  $W$-action.  Therefore for any $ w \in W$,
  $$i^\ast \circ i_\ast([\bar w]) = (\prod_{\alpha \in \Phi(\mathfrak g /
    \mathfrak p)} \alpha w ) \cdot  [\bar w],$$ 
  where $\Phi(\mathfrak g/ \mathfrak p) \subseteq \Phi$ denotes the
  subset of roots
  that appear in the diagonalization of the adjoint action on
  $\mathfrak g/ \mathfrak p$. 
  Notice that $\Phi(\mathfrak g/\mathfrak p)$ is a subset of
  the negative roots $\Phi^- := \Phi(\mathfrak g / \mathfrak b)$, for
  any choice of Borel subgroup $B \subseteq P$.  Hence, for the Chow class
  $\beta_w := \prod_{\alpha \in \Phi^- \setminus \Phi(\mathfrak g
    / \mathfrak p)} \alpha w \in A^\ast(BT)$,
  \begin{align*}
    i^\ast \circ i_\ast( \beta_w \cdot [\bar w]) & = (\prod_{\alpha \in \Phi^-}
  \alpha w )\cdot [\bar w].
  \end{align*}
  Since $\rootctop = \prod_{\alpha \in \Phi^-} \alpha$, 
  it follows that $\prod_{\alpha \in \Phi^-} \alpha w = \det(w)\cdot
  \rootctop$. 
  Therefore, $\rootctop \frown [X]$ is in the image of $i^\ast \circ i_\ast$:
  $$\sum_{\bar w \in W_P \backslash W} \rootctop \cdot [\bar w] =
  i^\ast \circ i_\ast \left( \sum \det(w) \cdot  \beta_w \cdot [\bar w]\right).$$
\end{proof}

The previous two lemmas
are directly analogous to results in equivariant
cohomology over $\C$ due to Brion (cf.~\cite[Lem.~2]{bri2}).  The
stratification of Theorem \ref{theorem-kirwan-strata-decomposition} is
more subtle for fields of characteristic $p$, and
we hence must stray from Brion's approach.
Even over 
$\C$, our method has the added benefit that most results hold
with integer, and not just rational, coefficients.

\begin{lemma}\label{lemma-group-surjection}
  Let $\pi:G \surject \bar G$ be a surjection of reductive groups over an
  algebraically closed field $\bar k$.  If $\bar k$ is characteristic
  $0$
  let $p =1$ and otherwise
  let $p := \Char k$.  If
  $f:Y \to U \subseteq \bar P \backslash \bar G$ is an equivariant
  morphism of $\bar G$-varieties, where $\bar P := \pi(P)$ and $U$ is
  the complement of the $T$-fixed points, then 
  $$\Delta = \ctop(\mathfrak g/ \mathfrak b) = 0 \in
  A_T^\ast(Y)_{\Z[\frac 1 p]}.$$
\end{lemma}
\begin{proof}
   Lemma \ref{rootctop-is-zero-from-localization-in-G/P-lemma}
   immediately implies that the class $\bar \Delta := \ctop(\bar
   {\mathfrak g} / \bar {\mathfrak b}) = 0 \in A_{\bar T}^\ast(Y)_{\Z}$,
   for the Lie algebras $\bar {\mathfrak g}$ and 
   $\bar {\mathfrak b}$ of $\bar G$ and the Borel subgroup $\bar B =
   \pi(B)$, respectively.   
   Therefore, it suffices to show that the pull-back of
   $\bar \Delta$ divides
   $p^N \cdot \Delta$ in the ring $A^\ast_T(Y)$ for some $N \gg 0$.
   This follows from \cite[5.1-2]{steinberg}, which states that
   $\pi^\ast: \chargp{\bar T} \to \chargp{T}$ is
   an injective homomorphism that takes 
   the roots of $\bar G$ into
   the set $\bigcup_{n=0}^\infty (p^n \cdot \Phi) \subseteq
   \chargp{T}$.

\end{proof}

From the previous lemma, we derive  that
$\rootctop$ is zero on the $T$-semi-stable locus of each stratum
$S_\beta \cap \semistable X T$.

\begin{lemma}\label{rootctop-is-zero-on-S-beta-semistable-lemma}  
  Let $X$ be a 
  smooth, projective $G$-variety with a $G$-linearized ample line bundle
  over $\bar k$.
  As an element
  of the $T$-equivariant operational Chow group of $S_\beta \cap
  \semistable X T$,
  $$\rootctop = 0 \in A^\ast_T(S_\beta \cap \semistable X T)_{\Z[\frac
      1 p]},$$
  where $p > 0$ is equal to the prime characteristic unless the
  characteristic is $0$, in which case $p = 1$.
\end{lemma}
\begin{proof}
  By Theorem
  \ref{theorem-kirwan-strata-decomposition}, there is a
  $G$-equivariant morphism 
  $\pi:S_\beta \to \bar P_\beta \backslash \bar G$ with
  $\pi\inv(\bar P_\beta) = Y_\beta \subseteq \unstable X T$.  
  The normalizer $N_T$ preserves the
  unstable locus $\unstable X T$, so in particular 
  $Y_\beta \cdot N_T \subseteq \unstable X T$.  Hence, the image of
  $\pi$ restricted to  
  $S_\beta \cap \semistable X T$
  is contained in $U$, the open complement in $\bar P_\beta \backslash
  \bar G$ of
  the $T$-fixed points $\bar P_\beta
  \backslash \bar P_\beta \cdot N_T$.  We conclude by invoking Lemma
  \ref{lemma-group-surjection}. 
\end{proof}

\subsection{Unions of strata} 
\label{integral-vanishing-of-rootctop-section}

We continue to assume that $X$ is smooth over an algebraically
closed field $\bar k$, and we extend the vanishing of $\rootctop$ on a
stratum to the vanishing over
the entire locus $\unstable X G \cap \semistable X T$.

\begin{proposition}\label{rootctop-is-zero-proposition}
  Let $X$ be a smooth, projective $G$-variety with a $G$-linearized ample line bundle
  over $\bar k$. 
  The operational Chow class $\rootctop$ annihilates every class in
  $A_\ast^T(\unstable X G \cap \semistable X T)_{\Z[\frac 1 p]}$,
  where $p > 0$ is equal to the prime characteristic unless the
  characteristic is $0$, in which case $p = 1$.
\end{proposition}
\begin{proof}
  Label $\mathbf B = \{\beta_1,\ldots, \beta_n\}$ so that 
  $\mathcal S_k := \bigcup_{i \leq k} S_{\beta_i}$ is a closed
  subvariety of $\unstable X G$, for each $1
  \leq k \leq n$; this is possible by 
  Theorem \ref{theorem-kirwan-strata-decomposition} (3).
  We proceed to inductively prove that $\rootctop$ annihilates every
  class in $A_\ast^T(\mathcal S_k \cap \semistable X T)_\Z$ for all $k
  = 1, \ldots, n$, with $k = n$ being the desired result.
  By Lemma
  \ref{rootctop-is-zero-on-S-beta-semistable-lemma}
  the assertion is true when $k = 1$ and $\mathcal S_1 = S_{\beta
    1}$ is an individual stratum.  
  Assume the result holds for $ 1 \leq k < n$.  By Proposition
  \ref{brions-presentation-of-T-equivariant-Chow-proposition}, it suffices
  to show that $\rootctop \frown [Y] = 0 \in
  A_\ast^T(\mathcal S_{k+1} \cap \semistable X T)$ for any $T$-invariant
  subvariety $Y \inject \mathcal S_{k+1}\cap \semistable X T$.
  Since $\mathcal S_k \inject \mathcal S_{k+1}$ is a
  closed immersion, the inductive
  hypothesis will imply
  $\rootctop \frown [ Y] = 0$ whenever $Y$ is completely contained in
  $\mathcal S_k$.   As $\mathcal S_{k+1} = S_{\beta_{k+1}} \cup
  \mathcal S_{k}$, the only remaining case is when $Y$
  intersects $S_{\beta_{k+1}}$ nontrivially.

  For the sake of clarity, denote $\beta := \beta_{k+1}$.
  Consider the birational map
  $\overline{S_\beta} \dashrightarrow \bar P_\beta\backslash \bar G$ defined on
  $S_\beta$ by 
  $S_\beta \cong Y_\beta \times_{\bar P_\beta} \bar G \to \bar P_\beta
  \backslash \bar G$
  as in Theorem \ref{theorem-kirwan-strata-decomposition}.
  Our strategy will be to partially resolve the locus of
  indeterminacy in the following manner:
  
  $$\xymatrix{ \overline{Y_\beta} \times_{\bar P_\beta} \bar G  \ar[d]_\pi
    \ar[rd]^{\tilde f} & \\ 
    \overline{S_\beta} \ar@{-->}[r]^f & \bar P_\beta \backslash \bar G.}$$
  The morphism $\pi$ is proper, since $\bar P_\beta \backslash \bar G$ is
  so, and is
  an isomorphism when restricted to the dense open $S_\beta \subseteq
  \overline{S_\beta}$.
  Moreover,  $\tilde f \inv(\bar P_\beta) = \overline {Y_{\beta}}$, and by
  $G$-invariance, 
  $$\tilde f\inv(\bar P_\beta W) = \overline {Y_{\beta}}
  \times_{\bar P_\beta} \bar P_\beta \cdot W.$$
  Since $N_T$ preserves $\unstable X T$ and $\overline
  {Y_{\beta}} \subseteq \unstable X T$, it follows that
  $\pi (\tilde f \inv(P_\beta W)) \subseteq \unstable X T$, showing
  that the above 
  diagram restricts to
  $$\xymatrix{ \pi\inv(\overline{S_\beta} \cap \semistable X T) \ar[d]_\pi
    \ar[rd]^{\tilde f} & \\ 
    \overline{S_\beta} \cap \semistable X T \ar@{-->}[r]^f & U,}$$
  where $U \subseteq \bar P_\beta \backslash \bar G$ is the open
  complement of the locus of $T$-fixed points.

  Let $\tilde Y$ denotes the strict transform of $Y$ under the birational
  morphism $\pi$.  From Lemma \ref{lemma-group-surjection} it follows
  that $\rootctop \frown [\tilde Y] = 0$.
  At the same time, $\pi_\ast(\rootctop \frown [\tilde Y]) = \rootctop
  \frown [Y]$ by the projection formula, so $\rootctop \frown [Y] = 0
  \in A_\ast^T(\overline{S_\beta} \cap \semistable X T)$.
  By the projection formula for the closed immersion 
  $\overline {S_{\beta}} \inject \mathcal S_{k+1}$, we conclude that
  $\rootctop \frown 
  [Y] = 0 \in A_\ast^T(\mathcal S_{k+1} \cap \semistable X T)$.
\end{proof}

\subsection{Arbitrary fields}\label{rootctop-is-zero-generally-section}
We now relax our previous assumption, henceforth allowing an arbitrary
base field $k$.  
The stratification of Theorem
\ref{theorem-kirwan-strata-decomposition}
 is constructed over an
algebraically closed field, so our previous arguments do not
immediately apply.  
However, if we weaken our statements by ignoring torsion, considering only
Chow groups with rational coefficients,  our
previous results quickly extend to this more general setting.
We outline the proof of the following well-established lemma
(cf.~\cite[Lem. 1A.3]{blo1}) for the reader's convenience.

\begin{lemma}\label{algebraic-field-extension-lemma}
  If $X$ is a variety over a field $k$, then any field extension $K/k$
  induces an injective morphism between Chow groups with 
  rational coefficients: $A_\ast(X)_\Q \inject A_\ast(X_K)_\Q$.
\end{lemma}
\begin{proof}
  If a field extension $E/k$ is the union of a directed system of sub-extensions
  $E_i/k$, 
  then $A_\ast(X_E) =
  \varinjlim  A_\ast(X_{E_i})$.  We may apply this result to
  the extension $K/k$ to reduce the proof
  to the two cases: $K/k$ is finite; or $K= k(x)$ for a transcendental
  element $x$.  Let $\phi: X_K \to X$ denote the base change morphism.

  If $K/k$ is finite,
  then $\phi$ is
  flat and  proper, and the composition $\phi_\ast \circ \phi^\ast$ is
  simply multiplication by $[K:k]$. This is an 
  isomorphism since coefficients are rational, and hence $\phi^\ast$
  is injective.  
  
  If $K = k(x)$, then $X_K$ is the generic fibre of the projection
  $\pi: X \times 
  \P^1_k \to \P^1_k$.  If $\psi: X \times \P^1_k \to X$ denotes the
  other projection and $j:X_K \to X \times \P^1_k$ is the inclusion of
  the generic fibre, then $\phi = \psi \circ j$. There is an
  isomorphism of Chow groups,  
  $$A_{i+1}(X \times \P^1_k)_\Q \cong A_{i}(X)_\Q \oplus A_{i+1}(X)_\Q
  \cdot t,$$
  where $t$ is the class associated to a fibre of $\pi$.
  The pull-back $\psi^\ast: A_i(X)_\Q \to A_{i+1}(X \times
  \P^1_k)_\Q$ is identified 
  with $\id \oplus 0: A_i(X)_\Q \to A_{i}(X)_\Q \oplus A_{i+1}(X)_\Q \cdot
  t$ (cf.~\cite[Thm.~3.3]{ful1}).
  As schemes,
  $$X_K = \varinjlim_{\emptyset \neq U \subseteq \P^1} X \times U,$$
  and there is an induced isomorphism on the level of Chow groups:
  $$A_i(X_K)_\Q \cong  \varinjlim_{\emptyset \neq U \subseteq \P^1}
  A_{i+1}(X\times U)_\Q \cong A_{i}(X)_\Q \oplus 0.$$
  Under this identification, the pull-back $j^\ast: A_{i+1}(X \times
 \P^1_k)_\Q \to A_{i}(X_K)_\Q$ is the projection
 $A_{i}(X)_\Q \oplus A_i(X)_\Q\cdot t \to A_i(X)_\Q$.
  From this description, it is clear that $\phi^\ast = j^\ast \circ
  \psi^\ast$ is an isomorphism, ergo injective.
\end{proof}

\begin{remark}\label{remark-integral-coefficients}
  When the extension $K/k$ is purely transcendental, the above proof
  actually shows that the pull-back $A_\ast(X) \inject A_\ast(X_K)$ is
  an injective group homomorphism between Chow groups with integer
  coefficients. 
\end{remark}

\begin{proposition}\label{rootctop-is-torsion-over-any-field-proposition}
  Let $X$ be a smooth, projective $G$-variety with a $G$-linearized ample line bundle
  over an arbitrary field $k$. The operational Chow class $\rootctop$
  annihilates every class in 
  $A_\ast^T(\unstable X G \cap \semistable X T)_\Q$.
\end{proposition}
\begin{proof}
Lemma \ref{algebraic-field-extension-lemma} reduces the proof to the
case of an algebraically closed base field, which is the content of Proposition
  \ref{rootctop-is-zero-proposition}.
\end{proof}

We state an immediate corollary that we will later
 extend to the case of
singular varieties (cf.~Prop. \ref{rootctop-is-numerically-zero-for-singular-X}). 

\begin{corollary}\label{rootctop-is-numerically-zero-for-smooth-X}
  Let $X$ be a smooth, projective $G$-variety over a field $k$, with a $G$-linearized ample line bundle
  for which $\semistable X T = \stable  X T$. 
  If $\tilde \alphanew_0, \tilde \alphanew_1 \in A_\ast(X \dd T)_\Q$ 
  are two lifts of the same class $\alphanew \in A_\ast(X \dd G)_\Q$,
  then for any operational Chow class $\arbclass \in A^\ast(X\dd T)_\Q$,
  $$\int_{X\dd T}  \arbclass \frown (\rootctop \frown \tilde
  \alphanew_0) =
  \int_{X\dd T}  \arbclass \frown (\rootctop \frown \tilde \alphanew_1).$$
\end{corollary}

\section{Singular varieties and strictly semi-stable points}\label{strictly-semi-stable-points-section}
\setcounter{equation}{0}

We wish to extend the results of 
\S\ref{section-rootctop-on-smooth-x}
 to the case
of singular $X$ by 
studying a compatible closed immersion $j: X \inject
\P(V)$, but several obstacles
must be overcome.
The first of which is the regrettable fact that
the push-forward map $j_\ast: A_\ast( \gitstack X T)_\Q \to
A_\ast( \gitstack {\P(V)} T)_\Q$ is not injective in general, thus
preventing an easy 
reduction to the smooth case. 
Suppose we attempt to circumvent
this problem by
limiting our ambitions to showing that any Chow class
$\alphanew \in A_\ast(X\dd G)_\Q$ gives rise to a well-defined 
\emph{numerical} equivalence
class $\rootctop \frown \tilde \alphanew$
independent of the choice of lift $\tilde \alphanew$.
Ideally the closed immersion $j$ would induce an injective map between
the groups of algebraic cycles modulo numerical equivalence, and we
would reduce our proof to the smooth case.
Unfortunately, there is the possibility that
$\semistable{\P(V)} T \neq \stable{\P(V)} T$,
and in this case there is no 
notion of numerical equivalence on the
Artin stack $\gitstack{(\P(V))}T$ (cf.~\cite{EGS}). Once again,
the na\"ive argument breaks down.

Our solution is to
build an auxiliary smooth, projective
$G$-variety $Y$ with a $G$-linearized ample line bundle for 
which $\semistable Y {T} = \stable Y {T}$,
and then to relate
integration on $X\dd G$ and $X\dd T$ to integration on $Y \dd G$ and
$Y \dd T$.
If one is guaranteed $G$-equivariant resolutions of
singularities (e.g. if $\Char k = 0$), then a result of
Reichstein \cite{rei1} generalizes the partial desingularizations
of Kirwan \cite{kir2} and produces such an auxiliary variety $Y$.  Since
resolution of singularities is still an open problem in
positive characteristic, we
provide an independent construction.
As an application of our method, we prove in \S
\ref{subsection-rootctop-is-numerically-zero} an extension
of Corollary \ref{rootctop-is-numerically-zero-for-smooth-X} to
singular varieties.

\subsection{Construction}

The auxiliary variety $Y$ will be defined as
an iterated product of the flag variety 
of right $B$-cosets, for a Borel subgroup $B$ containing the maximal torus
$T$.
The only delicate
point is the choice of a suitable $G$-linearized line bundle, which will rely upon
an understanding of equivariant line bundles over flag varieties.

\begin{lemma}\label{lemma-weight-line-bundle-linearization}
  For any weight $\chi \in \chargp T$ in the interior of the positive Weyl
  chamber, the 
  line bundle $\sL(\chi) := \A^1 \times_{\chi,B} G$
  is a very ample $G$-linearized line bundle on
  $B\backslash G$. 
   For any nontrivial subtorus $T' \subseteq T$ stabilizing some point 
   $Bg \in B\backslash G$, 
   the induced action on the fibre of $\sL(\chi)$ over $Bg$
  is by the restriction to $T'$ of the weight $( \chi \cdot w )$,
  where $w \in W$ corresponds to the Bruhat cell
  containing  $Bg$.
\end{lemma}
\begin{proof}
  That $\sL(\chi)$ is a very 
  ample line bundle on $B\backslash G$ is simply a rephrasing of the
  standard construction in representation theory of highest weight
  modules for dominant weights.  
  The Bruhat 
  decomposition states that $B\backslash G = \coprod_{w \in W} BwU$, for $U
  \subseteq B$ the maximal normal unipotent subgroup.  Furthermore, denoting by
  \makebox{$\dot w \in N_T$}
  a lift of
  the element $w \in W = T \backslash N_T$, there is an isomorphism
  $$\phi: B \times U_w 
  \to B \dot wU \subseteq G$$
  sending $(b, u)
  \mapsto b\dot w u$, where $U_w := U \cap (U^{-})^w$ is the
  intersection of $U$ with the $\dot w$-conjugate of the opposite
  unipotent subgroup, 
  $(U^-)^w := \dot w\inv U^- \dot w$
  (cf.~\cite[Thm.~14.12]{bor1}).
  Note that $U_w$ is  
  normalized by $T$ because both $U$ and $(U^-)^w$ are.
  We write $Bg = B\dot w u$ for some $w \in W$ and $u \in U_w$, and we
  assume $Bg$ is stabilized by the right action of $t \in T'$.
  Observe
  $$ \dot w u t = t^{\dot w \inv} \cdot \dot w \cdot u^t,$$
  where  $t^{\dot w \inv} := \dot w t \dot
  w\inv \in T \subseteq B$ and $u^t := t\inv u t \in U_w$. 
  Since $t$ stabilizes $Bg$, there exists some $b \in
  B$ such that $b \dot w u = t^{\dot w \inv} \cdot \dot w \cdot u^t$.
  Since $\phi$ is an isomorphism, it must be that
   $u = u^t$,  $b = t^{\dot w \inv}$, and
  $\dot w u t =  t^{\dot w \inv} \cdot \dot w u$.  From this equality, it
  is clear that $t$ acts
  on the fibre of $\sL(\chi)$ via multiplication by $\chi(t^{\dot w\inv}) =
  (\chi \cdot w)(t)$. 
\end{proof}

The following lemma will be the inductive step in the argument showing
that for any $G$-variety $Z$ with a $G$-linearized ample line bundle,
the variety 
$Y := Z \times 
(B\backslash G)^{\rank T}$  
admits a $G$-linearized ample line bundle for which
$\semistable Y T = \stable Y T$ and
all points in $Y$ projecting to a stable point of $Z$ are stable in $Y$.

\begin{lemma}\label{inductive-step-lemma}
  Let $\sL_Z$ be a $G$-linearized ample line bundle on a projective
  $G$-variety $Z$.
  If $r > 0$ is the maximum rank of a subtorus of $T$ that stabilizes
  a strictly semi-stable point 
  in $\strictlysemistable Z T$,  then
  there is a 
  character $\chi \in \chargp T$ 
  and an integer $n \gg 0$ so that the induced
  right $G$-linearized line bundle $\sL_Z^{\otimes n} \otimes
  \sL(\chi)$ on $Z \times B\backslash G$ is ample and
  has the following properties: 
  \begin{enumerate}
  \item The  subtori of $T$ that stabilize points in 
    $\strictlysemistable{(Z \times B\backslash G)}{T}$
    are at most rank $r -1$.
  \item 
    A point $p \in Z \times B\backslash G$ is stable
    (resp.~unstable)  whenever $\pi(p)$ is so, 
    where $\pi: Z \times B\backslash G \to Z$ is projection onto the
    first factor and the notion of stability is taken with respect to either
    the $G$- or the $T$-action.
  \end{enumerate}
\end{lemma}
\begin{proof}
  Let  $T_1,\ldots , T_n$ denote the positive-dimensional subtori
  of $T$ occurring as the connected components of stabilizers of points
  $z \in \strictlysemistable Z T$.
  Let $H_i := \Lambda^\ast(T/T_i) \subset \chargp T$ denote the
  subgroup of $T$-characters vanishing on $T_i$.
  Since $T_i$ is positive dimensional,  the inclusion $H_i \subsetneq
  \chargp T$ is strict.  Choose an integral weight
  $\chi$ in the interior of the positive Weyl
  chamber of $\chargp{T}$ avoiding the $W$-orbit of any $H_1, \ldots, H_n$.
  By Lemma \ref{lemma-weight-line-bundle-linearization}, $\sL(\chi)$ is
  a $G$-linearized ample line bundle on $B \backslash G$.

  By choosing $n$ large enough, one can derive from the
  Hilbert-Mumford criterion (Thm.~\ref{hilbert-mumford-criterion}),
  without too much effort,  that the linearized line bundle
  $\sL_Z^{\otimes n} \otimes \sL(\chi)$ satisfies:
  \begin{itemize}
    \renewcommand{\labelitemi}{\Cdot}
  \item $p \in Z \times B\backslash G$ is stable if $\pi(p)$ is
   stable,
  \item $p \in Z \times B\backslash G$ is unstable if $\pi(p)$ is
    unstable,
  \end{itemize}
  for stability with respect to either the $G$- or the $T$-action.
  This also follows immediately from a more general theorem of Reichstein
  (cf.~\cite[Thm.~2.1]{rei1}).
  As a consequence, any
  strictly semi-stable 
  point $p \in \semistable {(Z \times B\backslash G)} {T}$ must sit above
  a strictly 
  semi-stable point $\pi(p) \in \semistable Z T$. 

  If $T'$ is a torus stabilizing some point
  $(z,Bg) \in \strictlysemistable{(Z \times B\backslash G)} {T}$, then
  $z \in \strictlysemistable Z T$.
  Therefore $T' \subseteq T_i$ for some $1 \leq i \leq n$, and the
  weight of the action of $T'$ on the fibre of $\sL_Z$ over $z$ is
  $0$. Since $T'$ stabilizes $Bg$,  Lemma
  \ref{lemma-weight-line-bundle-linearization} implies 
  the weight of the action
  of $T'$ on the fibre of $\sL(\chi)$ over $Bg$ is 
  $ \chi \cdot w$ for some $w \in W$.
  Therefore, the weight of the  action of $T'$ on the fibre of
  $\sL_Z^{\otimes n}  \otimes \sL(\chi)$ over $(z, Bg)$ 
  is $0 +  (\chi\cdot w)|_{T'}$.
  Since $(z, Bg)$ is strictly semi-stable, 
  this weight must be $0$ and so $T' \subseteq \ker\left( (\chi\cdot
  w)|_{T_i} \right)$. 
  Our choice of $\chi$ was such that
  $( \chi \cdot w)|_{T_i} \neq 0 \in \chargp{T_i}$ for any $
  w\in W$, and thus the rank of $T'$ is at most the rank of 
  $\ker\left((\chi \cdot w)|_{T_i}\right )$, which is $r - 1$.
\end{proof}

\begin{proposition}\label{trivial-flag-variety-bundle-prop}
  If $Z$ is a projective $G$-variety with a $G$-linearized ample line bundle
  then for $r := \rank T$, the variety $Y := Z \times (B\backslash
  G)^{r}$ admits a $G$-linearized ample line bundle for which
  \begin{enumerate}[(i)]
  \item $\semistable Y {} = \stable Y {}$, and 
  \item $\stable Z {} \times (B\backslash G)^{r} \subseteq \semistable Y {}
    \subseteq \semistable Z {} \times (B\backslash G)^{r}$,
  \end{enumerate}
  for both $T$- and $G$- (semi-)stability.
\end{proposition}
\begin{proof}
  We prove the results for $T$-stability, and then 
  the result for $G$-stability will follow 
  (cf.~Thm.~\ref{hilbert-mumford-criterion}).
  Recursively applying Lemma \ref{inductive-step-lemma}, we obtain a
  $G$-linearized line bundle on $Z \times (B\backslash G)^{r}$ for which no
  $T$-strictly 
  semi-stable points have positive dimensional $T$-stabilizers.
  Hence, all $T$-semi-stable points are $T$-stable, proving (i).
  Lemma \ref{inductive-step-lemma}(2) guarantees (ii).
\end{proof}

\subsection{Integration}

Here we prove that the integration of Chow classes and their lifts on
GIT quotients of a
singular variety which equivariantly embeds into a nonsingular variety
can be related to the 
integration of Chow classes and lifts on an auxiliary variety
constructed as in the previous subsection.
We begin with an easy lemma describing how lifting Chow classes
commutes with proper pushes-forward.

\begin{lemma}\label{lemma-lifts-commute-with-proper-pushforward}
  If $\pi: X \to Y$ is an equivariant proper morphism of
  $G$-varieties admitting $G$-linearized ample line bundles so that
  $\semistable X T = \stable X T$, $\semistable
  Y T = \stable Y T$, and $\pi\inv(\semistable Y T) = \semistable X T$, then
  for any Chow class $\alphanew \in A_\ast(X \dd G)_\Q$ and any lift 
  $\tilde  \alphanew \in 
  A_\ast(X \dd T)_\Q$, the push-forward $\pi_\ast\tilde \alphanew \in
  A_\ast(Y \dd T)_\Q$ is a
  lift of the Chow class $\pi_\ast\alphanew \in A_\ast(Y \dd G)_\Q$.
\end{lemma}
\begin{proof}
   The Hilbert-Mumford criterion
  (Thm.~\ref{hilbert-mumford-criterion}) implies
  that $\pi\inv(\semistable Y G) = \semistable X G$, yielding
   the following fibre square,
$$\xymatrix{\semistable X G \ar@{^(->}[r]^{i_X} \ar[d]_\pi & \semistable X T
    \ar[d]_\pi\\
    \semistable Y G \ar@{^(->}[r]^{i_Y} & \semistable Y T,\\}$$
  where the horizontal arrows are open
  immersions, and the vertical arrows are proper morphisms.
  By Definition \ref{definition-of-lift}, the Chow classes
  $\alphanew, \tilde \alphanew$ correspond to
  classes $\alphanew \in A_\ast([\semistable X G/ G])_\Q$ and
  $\tilde \alphanew \in 
  A_\ast([\semistable X T / T])_\Q$ satisfying $i_X^\ast \tilde \alphanew =
  f_X^\ast\alphanew$, where $f_X: [\semistable X G/G] \to [\semistable X
    G/T]$ is induced by the inclusion $T \subseteq G$.  

  All that must be verified is whether
  $$i_Y^\ast (\pi_\ast \tilde \alphanew) \stackrel ? =  f_Y^\ast (\pi_\ast
  \alphanew),$$
  where the morphism $f_Y: \gitstack Y G \to [\semistable Y G/T]$
  is induced by the inclusion $T \subseteq G$.
  Since the above diagram is a fibre square, there is the
  commutativity 
  $i_Y^\ast \circ \pi_\ast = \pi_\ast \circ i_X^\ast$
  (cf.~\cite[Prop. 1.7]{ful1}).
  Therefore, 
  \begin{align*}
    i_Y^\ast (\pi_\ast \tilde \alphanew) & 
    = \pi_\ast( i_X^\ast \tilde \alphanew)\\
    & = \pi_\ast(f_X ^\ast \alphanew), 
  \end{align*}
  since $\tilde \alphanew$ is a lift of $\alphanew$.  It remains 
  to show that $\pi_\ast \circ f_X^\ast = f_Y^\ast
  \circ \pi_\ast$, which follows from the fact that the 
  commutative diagram of Deligne-Mumford stacks,
  $$\xymatrix{ [\semistable X G/T] \ar[r]^{f_X} \ar[d]_{\pi_T} & \gitstack
    X G \ar[d]_{\pi_G}\\
    [\semistable Y G/T] \ar[r]^{f_Y} & \gitstack Y G,}\\$$
  is a fibre square, as demonstrated by the following lemma.
\end{proof}

\begin{lemma}\label{lemma-fibre-square-of-stacks}
  If $H \subseteq G$ is a closed subgroup of $G$ and $f: Y \to Z$ is
  an equivariant morphism of $G$-varieties, then $f$ induces the
  following 2-Cartesian diagram of quotient stacks:
\begin{equation}\label{diagram-fibre}
\xymatrix{ [Y/{H}] \ar[r] \ar[d] & 
    [Y/ G] \ar[d]\\
    [ Z/ H] \ar[r] &  [Z/G].}
\end{equation}
\end{lemma}
\begin{proof}
  To begin, we recall that for a closed subgroup $H \subseteq G$ and
  any $G$-variety $U$,
  there is a Morita equivalence, corresponding to the isomorphism of
  quotients stacks 
  \begin{equation}\label{morita-equivalence}
    [U/H] \cong [(U \times_H G)/G],
  \end{equation}
  given by the operation $-\times_H G$ that converts an $H$-torsor
  mapping to $U$ into a $G$-torsors mapping to $U \times_H G$.  

  Note how the automorphism of $U \times G$ given by the rule $(u, g)
  \mapsto (ug, g\inv)$ intertwines the right action of $H$ on the $G$-factor
  with the right action given by the rule $(u, g)\cdot h :=
  (uh, h\inv g)$.  This induces a $G$-equivariant isomorphism 
  $$U \times G/H \cong U \times_H G,$$
  where the $G$-action on $U \times G/H$ is given by
  the rule
  $(u, gH)\cdot \hat g = (u\hat g, \hat g\inv gH)$,
  and the $G$-action on $U \times_H G$ by the rule
  $(u, g) \cdot \hat g = (u, g \hat g)$.
  Combined with \eqref{morita-equivalence}, this gives an isomorphism
  of quotient stacks
  \begin{equation}\label{equation-quotient-stack-isomorphism}
    [(U\times G/H)/G] \cong [U/H].
  \end{equation}
  Under this identification,
  the natural  morphism $[U/H] \to [U/G]$ induced by the inclusion
  $H \subseteq G$  corresponds to the morphism 
  $[(U \times G/H)/G] \to [U/G]$   
  induced by the $G$-equivariant projection $U \times G/H \to U$.  
  Finally, since the $G$-equivariant diagram
  $$\xymatrix{ Y \times G/H \ar[r]\ar[d] & Y  \ar[d]\\
    X \times G/H \ar[r] & X \\}$$
  is Cartesian, so too will be the corresponding
  diagram of quotient stacks in the $2$-category of stacks. 
\end{proof}

We are now prepared to prove the main result of this subsection.

\begin{proposition}\label{prop-comparing-integrals-on-auxiliary-space}
  Let $j:X \inject Z$ be an equivariant embedding of $G$-varieties, and
  let $\sL$ be a $G$-linearized ample line bundle on $Z$.  There is a
  $G$-linearized ample line bundle on $Y := Z \times (B\backslash  G)^{r}$, where $r
  := \rank T$, for which $\semistable Y T = \stable Y T$ and
  furthermore,  if  $\semistable X T = \stable X T$ for the induced
  $G$-linearized line bundle  $j^\ast \sL$ on $X$, then  for any $\alphanew \in
  A_\ast(X \dd G)_\Q$ with lift $\tilde \alphanew \in A_\ast( X \dd
  T)_\Q$, there is a class $\betanew \in A_\ast(Y \dd
  G)_\Q$ with lift   $\tilde \betanew \in A_\ast(Y \dd T)_\Q$ so that 
  \begin{enumerate}
  \item $\int_{X\dd G} \alphanew = \int_{Y\dd G}
    \betanew,$ \hspace{0.5ex} and 
  \item $\int_{X\dd T} \arbclass \frown (\rootctop \frown
    \tilde \alphanew)  = 
    \int_{Y\dd T} \arbclass \frown (\rootctop \frown \tilde  \betanew),
    \hspace{1ex} \forall \arbclass \in A^\ast(X\dd T)_\Q.$
  \end{enumerate}

\end{proposition}
\begin{proof}
  Let $\sM$ denote the $G$-linearized ample line bundle on $Y$
  prescribed in Proposition
  \ref{trivial-flag-variety-bundle-prop}. 
  Define the variety $\hat X := X \times (B \backslash G)^{r}$ and the
  embedding $\hat j = j \times \id: \hat X \inject Y$.  Endow
  the line bundle $\hat j^\ast \sM$ on $\hat X$ with the
  $G$-linearization induced from that on $\sM$.  Proposition
  \ref{trivial-flag-variety-bundle-prop}  guarantees that  
  $\semistable Y {T} = \stable Y {T}$, and hence $\semistable {\hat X}
  {T} = \stable {\hat X} {T}$.  Moreover, it also shows that $\hat j \inv(\semistable Y {T}) = \semistable {\hat X} {T}$ and
  $\pi\inv(\semistable X {T}) = \semistable {\hat X}{T}$, for $\pi: \hat X
  \to X$ the projection morphism.
  Since $\pi|_{\semistable X G}$ is surjective, there exists a class $\rho \in
  A_\ast(\hat X \dd G)_\Q$ such that $\pi_\ast \rho = \alphanew$.
  Let $\tilde \rho \in A_\ast(\hat X \dd T)_\Q$ be a lift of the Chow
  class $\rho$ for which $\pi_\ast(\tilde \rho) = \tilde \alphanew$,
  which is possible because $\pi|_{\semistable X T}$ is surjective.  
  By Lemma \ref{lemma-lifts-commute-with-proper-pushforward}, $\tilde
  \betanew := \hat j_\ast\tilde \rho$ is a lift of the class
  $\betanew := \hat j_\ast\rho$.  

  By the functoriality of the Chow groups under proper
  push-forward, 
  the commutative diagram,
  $$\xymatrix{\hat X \dd G  \ar[r]^{\hat j} \ar[d]_\pi 
    & Y \dd G \ar[d]\\
   X \dd G \ar[r] & \Spec k,}$$
  shows that $\int_{X \dd G} \alphanew = \int_{Y \dd G} \tau$, proving
  (1).
  The equality in (2)
  follows in the same way from an analogous diagram involving
  $T$-quotients, recalling that the
  projection-formula guarantees the compatibility of the operational Chow
  class $\arbclass \frown \rootctop$ under pushes-forward.
\end{proof}

\subsection{Application}\label{subsection-rootctop-is-numerically-zero}

\begin{proposition}\label{rootctop-is-numerically-zero-for-singular-X}
  Let $X$ be a projective variety  over
  a field $k$ with a $G$-linearized ample line bundle for which
  $\semistable X T = \stable  X T$. 
  If $\tilde \alphanew_0, \tilde \alphanew_1 \in A_\ast(X \dd T)_\Q$ 
  are two lifts of the same class $\alphanew \in A_\ast(X \dd G)_\Q$,
  then for any operational Chow class $\arbclass \in A^\ast(X\dd T)_\Q$,
  $$\int_{X\dd T} \arbclass \frown (\rootctop \frown \tilde
  \alphanew_0) =
  \int_{X\dd T} \arbclass \frown (\rootctop \frown \tilde \alphanew_1).$$
\end{proposition}

\begin{proof}
Embed the singular $X$ into the smooth variety $\P(V)$ via some high
tensor power of the given $G$-linearized line bundle.  Construct the smooth
$G$-variety $Y$ as in Proposition \ref{prop-comparing-integrals-on-auxiliary-space}
by taking $Z := \P(V)$.  The proposition
guarantees that any two lifts 
$\tilde \alphanew_0$ and $\tilde \alphanew_1$ of a class $\alphanew \in
A_\ast(X\dd G)_\Q$ have analogues $\tilde \betanew_0$ and $\tilde
\betanew_1$ both lifting a class $\betanew \in A_\ast(Y \dd G)_\Q$ and
satisfying   
$$\int_{X\dd T} \arbclass \frown (\rootctop \frown \tilde \alphanew_i)
= \int_{Y\dd T} \arbclass \frown (\rootctop
\frown \tilde \betanew_i),$$
for $i = 0,1$.  The result immediately follows from the smooth case
(cf.~Cor. \ref{rootctop-is-numerically-zero-for-smooth-X}).
\end{proof}

\section{Invariance of the GIT integration ratio}
\label{section-independence-of-GIT-ratio}
\setcounter{equation}{0}

The goal of this section is to prove Theorem
\ref{GIT-integral-ratio-is-invariant-of-G-theorem}, that for a
projective $G$-variety
$X$ over a field $k$ and a $G$-linearized ample line bundle for which
$\semistable X T = \stable X T$,
the following ratio 
is an invariant of the group $G$:
\begin{definition}\label{definition-GIT-integral-ratio}
  Assume $X$ is a projective $G$-variety with a $G$-linearized ample
  line bundle $\sL$
  for which $\semistable X T = \stable X
  T$, and  $\alphanew \in A_0(X \dd G)$ is a $0$-cycle
  satisfying $\int_{X\dd G} \alphanew \neq 0$.  We define the
  \emph{GIT integration ratio} to be
  $$r_{G,T}^{X,\alphanew} :=
  \frac{ \int_{X\dd T} \ctop(\mathfrak g / \mathfrak t) \frown \tilde  \alphanew}
       { \int_{X\dd G} \alphanew},$$
  where $\tilde \alphanew$ is some lift of the class $\alphanew$.  
\end{definition}

Note that the GIT integration ratio is well-defined, as can be seen by
taking $\arbclass = (-1)^{|\Phi|/2} \rootctop$ and applying Proposition
\ref{rootctop-is-numerically-zero-for-singular-X}.
We shall next prove in stages that $r_{G,T}^{X,\alphanew}$
is independent of the choice of 
Chow class $\alphanew \in A_0(X\dd G)_\Q$, the choice
of maximal torus $T$, and even the choices of $G$-linearized ample
line bundle $\sL$ and projective $G$-variety $X$. 

\subsection{Chow class}
\label{subsection-independence-on-chow-class}

\begin{lemma}\label{independent-of-class-DM-case-lemma}
  The GIT integration ratio $r^{X}_{G,T} := r^{X,\alphanew}_{G,T}$ is
  independent of the choice of Chow class $\alphanew\in A_0(X\dd G)_\Q$.
\end{lemma}
\begin{proof}
  The definition of $r^{X,\alphanew}_{G,T}$ is
  independent of the algebraic equivalence class of $\alphanew$, since
  numerical equivalence is coarser than algebraic equivalence.  Let
  $B_\ast(-)$ denote the quotient of the Chow group $A_\ast(-)$ by the
  relation of algebraic equivalence (cf.~\cite[\S 10.3]{ful1}).
  All connected projective schemes are algebraically
  connected, i.e. there is a connected chain of (possibly singular)
  curves connecting any two closed points. Therefore,
  $B_0(X\dd G)_\Q = \Q$, and the result follows since
  $r^{X,\alphanew}_{G,T}$ is invariant under the multiplication of
  $\alphanew$ by a nonzero scalar.
\end{proof}

\subsection{Maximal torus}

We begin by reducing to the case where $k = \bar k$ is
an algebraically closed field.

\begin{lemma}\label{lemma-independent-under-field-extension}
  Let $X$ be a projective $G$-variety over a field
  $k$ equipped with a $G$-linearized ample line bundle $\sL$ for which
  $\semistable X T = \stable X T$.  
  For any field extension $K/k$, the GIT integration ratios are equal,
  $$r_{G,T}^X = r_{G_K,T_K}^{X_K},$$
  where $X_K$, $G_K$, and $T_K$ denote the base changes from $k$ to $K$
  of $X$, $G$, and $T$, respectively.
\end{lemma}
\begin{proof}
 By  \cite[Prop. 1.14]{GIT}, 
 $\semistable X G \times K = \semistable {(X_K)} {G_K}$ for the
 induced $G_K$-linearization on $\sL_K$ over
 $X_K$ (and
 similarly with $T$ replacing $G$).  
  Since field extensions are flat, 
 taking the uniform categorical quotient commutes
 with extending the field, so 
 $(X\dd G)_K \cong X_K\dd G_K$ and $(X\dd T)_K \cong X_K\dd T_K$.
 The result can then be deduced from the fact that the degree of a Chow
 class is invariant 
 under field extension (cf.~\cite[Ex. 6.2.9]{ful1}).
\end{proof}

\begin{lemma}\label{independent-of-maximal-torus-lemma}
  The GIT integration ratio $r_G^X := r_{G,T}^X$ does not depend on the
  choice of maximal torus $T$.
\end{lemma}
\begin{proof}
  By Lemma \ref{lemma-independent-under-field-extension}, we may
  assume that $k = \bar k$.
  For any two maximal tori $T, T' \subseteq G$,
  there exists some $g \in G(\bar k)$ such that  $T' =  gTg\inv$.
  By assumption, $G$ acts linearly on
  the projective variety $X$.
  Consider the map $\psi:T \to T'$ given by
  $t \mapsto g t g\inv$, and the map $\Psi:X \to X$ given by $x
  \mapsto x\cdot g\inv$.  The pair of maps $(\psi,\Psi)$ show that
  the actions $\sigma: X \times T \to X$ and $\sigma': X
  \times T' \to X$ are isomorphic: $\Psi(x)\cdot \psi(t) = \Psi(x\cdot
  t)$.  By the Hilbert-Mumford numerical
  criterion (Thm.~\ref{hilbert-mumford-criterion}), 
  $\semistable X {T'} =  \Psi(\semistable X T)$, and
  hence  the following square is commutative, 
  $$\xymatrix { \semistable X G \ar[r]^{ \Psi} \ar@{_(->}[d]_{i_{T}} &
    \semistable 
    X G \ar@{_(->}[d]_{i_{T'}}\\
    \semistable X T \ar[r]^{ \Psi} & \semistable X {T'}, }$$ 
  where the vertical arrows are open immersions and the horizontal arrows
  are isomorphisms.
  From this diagram, it is clear that 
  for any lift $\tilde \alphanew \in A_\ast(X \dd T)_\Q$ of a Chow
  class $\alphanew \in  A_0(X\dd G)_\Q$, 
  the push-forward $ \Psi_\ast \tilde \alphanew \in A_\ast(X \dd T')_\Q$ 
  is a lift of $\Psi_\ast \alphanew = \alphanew$.
  We use the pairs 
  $(\alphanew, \tilde \alphanew)$ and $( \alphanew,  \Psi_\ast
  \tilde \alphanew)$ to
  compute the GIT integration ratios.  Since $ \Psi$ is an
  isomorphism, the classes $\ctop(\mathfrak g / \mathfrak t) \frown \tilde \alphanew$ and 
  $\Psi_\ast (\ctop(\mathfrak g / \mathfrak t) \frown \tilde
  \alphanew) = \ctop(\mathfrak g / \mathfrak t') \frown \Psi_\ast 
  \tilde \alphanew$ have the same degree, where $\mathfrak t '$
  denotes the Lie algebra of $T'$.   Therefore, the ratios
  $r_{G,T}^{X}$ and $r_{G,T'}^{X}$ are
  equal. 
\end{proof}

\subsection{Linearized line bundle and variety}

\begin{lemma}
  The GIT integration ratio $r_G := r_{G}^X$ does not depend on either the
  choice of $G$-linearized ample line bundle $\sL$ or the projective $G$-variety $X$.  
\end{lemma}

\begin{proof}
  For each $i = 1,2$, let $X_i$ be a
  projective $G$-variety with a $G$-linearized ample line bundle
  $\sL_i$ 
  for which $\semistable {(X_i)} T = \stable {(X_i)} T$.  
  Some high tensor power of $\sL_i$ 
  defines a $G$-equivariant embedding $j_i: X_i \inject \P(V_i)$ for some
  $G$-representation $V_i$.
  Consequently, one derives an equivariant embedding $\hat j_i: X_i \inject \P(V_1
  \oplus V_2)$
  from the embedding $j_i$, simply by setting the
  extraneous coordinates to 0. 
  Moreover, the embeddings $\hat j_i$ are compatible with the
  $G$-linearization on the
  line bundle $\OO(1)$  on  $\P(V_1 \oplus V_2)$,  induced from the direct sum
  representation of $G$ on $V_1 \oplus V_2$, and the given
  $G$-linearizations on the line bundles $\sL_i$.
  By Proposition \ref{prop-comparing-integrals-on-auxiliary-space},
  for any classes $\alphanew_i \in A_0(X_i\dd G)_\Q$, there are classes
  $\betanew_i \in A_0(Y\dd G)_\Q$, for some smooth, projective
  $G$-variety $Y$ with a $G$-linearized ample line bundle for which 
  $\semistable Y T = \stable Y T \neq \emptyset$
  and $r_{G}^{Y,\betanew_i} = r_{G}^{X_i,\alphanew_i}$.
  By Lemma \ref{independent-of-class-DM-case-lemma}, 
  $r_{G}^{Y,\betanew_1} = r_{G}^{Y,\betanew_2}$ and therefore
  $r_{G}^{X_1, \alphanew_1} =  r_G^{X_2,\alphanew_2}$. 
\end{proof}

 The above lemmas combine to demonstrate
 Theorem~\ref{GIT-integral-ratio-is-invariant-of-G-theorem}.

\section{Direct products and central isogenies} 
\label{section-functorial-properties}
\setcounter{equation}{0}
\renewcommand{\theequation}{\arabic{section}.\arabic{equation}}

In this section, we prove that the GIT integration ratio behaves well
with respect to the group operations of direct product and central
isogeny.  The results of this section combine to prove Theorem
\ref{GIT-integral-ratio-decomposes-multiplicatively}.

\begin{proposition}\label{product-of-groups-GIT-ratio-lemma}
  If $G_1, G_2$ are two reductive groups over a field $k$, then
  $r_{G_1\times G_2} =  r_{G_1} \cdot r_{G_2}$.
\end{proposition}
\begin{proof}
  For each $i = 1,2$,  choose a projective $G_i$-variety $X_i$ and a
  $G_i$-linearized ample line bundle $\sL_i$.  Let $T_i \subseteq G_i$ denote a maximal torus.
  Clearly $G_1  \times  G_2$ acts on $X_1 \times X_2$, and $\sL_1
  \otimes \sL_2$ is a $G_1 \times G_2$-linearized ample line bundle
  for which $\semistable {(X_1 \times X_2)} {G_1 \times G_2} =
  \semistable {(X_1)} {G_1} \times \semistable {(X_2)} {G_2}$.  
  Let $\alphanew_i \in A_0^{G_i}( \semistable {(X_i)} {G_i})$, and consider
  $\alphanew:=\alphanew_1\times \alphanew_2 \in A_0^{G_1 \times G_2}(
  \semistable{(X_1)} {G_1} \times \semistable{(X_2)}{G_2}).$  Also,
  take $\tilde \alphanew := \tilde \alphanew_1 \times \tilde \alphanew_2$,
  where each $\tilde \alphanew_i$ lifts $\alphanew_i$.
  By Theorem~\ref{GIT-integral-ratio-is-invariant-of-G-theorem},
  it suffices to 
  calculate the GIT integration ratio for $G_1 \times G_2$ using the
  classes $\alphanew$ and $\tilde \alphanew$.  The degree of
  a product of two classes is the product of the degrees, and so
  the result follows since the isomorphism
  $[\Spec k/T] \cong [\Spec k /T_1] 
  \times_k [\Spec k/T_2]$ identifies
  $\ctop(\mathfrak g_1 / \mathfrak t_1) \times \ctop(\mathfrak
  g_2/\mathfrak t_2)$ with $\ctop(\mathfrak g / \mathfrak t)$.
\end{proof}

Now we prove that $r_G$ is invariant under
central isogeny.  This will completing the proof of Theorem
\ref{GIT-integral-ratio-decomposes-multiplicatively}  (since the GIT
integration ratio of a torus is clearly $1$).

\begin{proposition}\label{prop-central-extension}
  If $G \surject \bar G$ is a central isogeny of connected reductive groups, 
  then $r_{ G} = r_{\bar G}$.  
\end{proposition}
\begin{proof}
  Since the kernel of $G \surject G$ is a finite group scheme of order
  $d < \infty$, we may use the 
  same $\bar G$-variety $X$ on which to calculate both ratios $r_G$
  and $r_{\bar G}$.  These ratio will therefore agree, since the
  numerator and the denominator of the ratio $r_G$ will only differ
  from those of $r_{\bar G}$ by a factor of $d$.
\end{proof}

\section{Groups of type $\mathbf{A}_n$}\label{section-calculation}
\setcounter{equation}{0}

We compute the GIT integration ratio $r_G$
for $G = PGL(n)$, and use this to prove Corollary \ref{main-theorem}.

\begin{proposition}\label{proposition-calculation-for-pgln}
  The GIT integration ratio for the group $G = PGL(n)$ defined over any
  field is $$r_G = |W| = n!.$$
\end{proposition}
\begin{proof}
  Let $GL(n)$ act on $\mathbf M_{n}$, the vector space of $n \times n$
  matrices with $k$-valued entries, via right multiplication of
  matrices. 
  This induces a dual right-action of $GL(n)$ on $\mathbf M_n
  ^\ast$ and hence right-actions of $GL(n)$, $SL(n)$, and $PGL(n)$ on $\P(\mathbf
  M_n)$, the projective space of lines in $\mathbf M_{n}^\ast$. 
  Choose the $PGL(n)$-linearization
  on $\OO_{\P(\mathbf M_{n})}(n)$ induced from  
  these representations on $\mathbf M_n$.
  Let $T \subseteq GL(n)$, $S
  \subseteq SL(n)$,  and $\bar T \subseteq PGL(n)$ denote the diagonal
  maximal tori. Because of the isogeny $SL(n) \surject PGL(n)$, the 
  $PGL(n)$-stability loci (resp.~$\bar T$-stability loci) are equal to
  the analogous $SL(n)$-stability loci (resp.~$S$-stability loci), 
  which we proceed to describe.

  A basis of $\mathbf M_n$ is given by the matrices $e_{ij}$,
  each defined by its unique nonzero entry of $1$ in the $(i,j)$th
  position.    
  Moreover, $e_{ij}$ is a weight vector of weight $\chi_j
  \in \chargp{T}$, with $\chi_j$ defined by the rule
  $$\begin{pmatrix}  t_1 &  
&
  & \\
     & t_2 &  & \\
     & 
 &
 \ddots &  \\
     & 
&
  & t_n\\
  \end{pmatrix}  \mapsto t_j.$$
  Notice that $\chi_1,\ldots, \chi_{n}$ is a basis of the
  character group $\chargp {T}$, and $\chargp{S}$ is spanned by the
  restriction of these characters $\chi_i|_S$, which are only subject to
  the relation  $\chi_n|_S = - \sum_{i=1}^{n-1} \chi_i|_S$.
  Therefore, the characters $\chi_1|_S,\ldots, \chi_n|_S$ form the vertices of a
  simplex centered at the origin in the vector space
  $\chargp{S}\otimes \Q \cong \Q^{n-1}$. 
  From the Hilbert-Mumford criterion
  (Thm.~\ref{hilbert-mumford-criterion}), one quickly concludes that 
  the unstable locus $\unstable {\P(\mathbf M_{n})}{\bar T}$
  is the set of all points $x \in \P(\mathbf M_n)$ such that the
  matrix $e_{ij}(x)$
  has a column with all entries $0$; all
  other points are $\bar T$-stable. 
  Such stable points have trivial $\bar T$-stabilizers, and so the 
  $\bar T$-quotient is easily seen to be
  $$ \P(\mathbf M_n)\dd \bar T \cong \gitstack {\P(\mathbf M_n)}{\bar
    T} \cong (\P^{n-1})^n.$$
  Furthermore, if we denote by
  $\mathbf M_n^\circ \cong (\A^{n}\setminus 0)^n$ the preimage in
  $\mathbf M_n \setminus 0$ of $\semistable{\P(\mathbf M_n)}{\bar T}$,
  there is also an identification
  $$ \P(\mathbf M_n) \dd \bar T \cong [\mathbf M_n^\circ / T],$$
  From this later description, we see immediately that the rational
  Chow ring is given as an $A^\ast(BT)_\Q \cong \Sym_\Q^\ast \chargp{T}$-algebra
  by 
  $$A^\ast(\P(\mathbf M_n)\dd \bar T )_\Q \cong
  \Q[\chi_1,\ldots, \chi_n]/(\chi_1^n,\ldots, \chi_n^n).$$
  In this ring, the class of a point is clearly $\prod_{i=1}^n
  \chi_i^{n-1}$, and therefore the degree of a Chow cycle is given by
  the coefficient of this monomial.
  
  The $PGL(n)$-stable locus
  then comprises the set of $x \in \P(\mathbf M_n)$ such that the
  matrix $e_{ij}(x)$ is of full-rank; this comprises the
  dense $PGL(n)$ orbit given by the inclusion $PGL(n) \subseteq
  \P(\mathbf M_n)$.
  The stabilizer of this orbit is trivial,
  and hence
  $$  \P(\mathbf M_{n})\dd PGL(n) \cong \gitstack {\P(\mathbf
    M_{n})}{PGL(n)} \cong \Spec k.$$
  The Chow ring of this variety is simply $\Q$ in degree 0, and by
  choosing $\alphanew = 1$, the fundamental class, the computation of
  the GIT integration ratio $r_G$ is reduced to evaluating
  $r_{G} = \int_{\P(\mathbf M_n)\dd T} \ctop(\mathfrak g / \mathfrak t)$.

  The class $\ctop(\mathfrak g / \mathfrak t) \in A^\ast(B \bar T)$ is the
  product of all the roots of $PGL(n)$; these are of the form
  $\alpha_{ij} := \chi_i - \chi_j \in \Sym^\ast \chargp{T} \cong
  A^\ast(B \bar T)$, for $1 \leq i \neq
  j\leq n$.  Therefore $r_G$ equals the coefficient of the monomial
  $\prod_{i=1}^n \chi_i^{n-1}$ in the expansion of $\prod_{i \neq j} (\chi_i -
  \chi_j)$.
  We can compute this coefficient as follows.
  The class $\ctop(\mathfrak g / \mathfrak t)$ may be alternatively
  expressed as 
  $$\prod_{i\neq j}(\chi_i - \chi_j) = (-1)^{n(n-1)/2} (\det M_V)^2,$$  
  where $\det M_V$ is the determinant of the Vandermonde matrix
$$
  M_V :=
  \begin{pmatrix}
    1 & \chi_1 & \chi_1^2 & \cdots & \chi_1^{n-1}\\
    1 & \chi_2 & \chi_2^2 & \cdots & \chi_2^{n-1}\\
    \vdots & \vdots & \vdots & \ddots & \vdots\\
    1 & \chi_n & \chi_n^2 & \cdots & \chi_n^{n-1}\\
  \end{pmatrix}.
  $$
  By definition, $\det M_V = \sum_{\sigma \in S_n} \mathrm{sgn}(\sigma)
  \prod_{i=1}^{n} \chi_i^{\sigma(i)-1}$.  In the
  Chow ring $A^\ast(\P(\mathbf M_n)\dd \bar T )_\Q$, the products of 
  monomials of the form $m_\sigma := \prod_{i=1}^n \chi_i^{\sigma(i) - 1}$
  for $\sigma \in S_n$ are defined by the rule: 
  $$m_\sigma \cdot m_{\sigma'} =
  \left \{ 
  \begin{array}{l l}
    \prod_{i=1}^n \chi_i^{n-1} & : \sigma(j) + \sigma'(j) = n+1; ~\forall~1 \leq j
    \leq n;\\
    0 & : \textrm{ otherwise.}
  \end{array}\right.$$
  If $w_0 := (1~n)(2~n-1)\cdots(\lceil n/2 \rceil ~ \lceil (n+1)/2
  \rceil) \in S_n$ denote the longest element of the Weyl group $W = S_n$,
  then for each 
  $\sigma \in S_n$, the permutation  $\sigma'$ defined as the composition
  $\sigma' := w_0 \circ \sigma$ is the 
  unique element of $S_n$ for which  $m_\sigma \cdot m_{\sigma'}\neq 0$. 
  For such pairs $(\sigma, \sigma')$, the product of the signs satisfies
  $\mathrm{sgn}(\sigma)\cdot \mathrm{sgn}(\sigma') = \mathrm{sgn}(w_0) = (-1)^{(n^2 - n)/2}$.
  Therefore, 
  \begin{align*}
    \ctop(\mathfrak g / \mathfrak t) & = (-1)^{(n^2 - n)/2} \cdot \sum_{\sigma \in S_n} (-1)^{(n^2 -
      n)/2} \prod_{i=1}^n \chi_i^{n-1}\\
    & = n! \cdot \prod_{i=1}^n \chi_i^{n-1}.
  \end{align*}
  Thus, $r_G = n! = |W|$. 
\end{proof}

Having proved the above proposition, the proof of Corollary
\ref{main-theorem} is now anticlimactic: 

\begin{proof}
  [Proof of Cor.~\ref{main-theorem}]
  Combine Theorems \ref{GIT-integral-ratio-is-invariant-of-G-theorem} and
  \ref{GIT-integral-ratio-decomposes-multiplicatively},
  and Proposition \ref{proposition-calculation-for-pgln}.
\end{proof}

\section{Groups of other types}
\label{section-final-remarks}
\setcounter{equation}{0}

We conclude 
with a discussion of how to generalize Corollary
\ref{main-theorem} to arbitrary connected reductive groups, that is, how
to prove Amplification \ref{amplification-general-result}.  The argument
will rely on Martin's original result \cite[Thm.~B$'$]{mar1}, and the discovery
of an independent proof is left as an open question.

\setcounter{repeatcounter}{6}
\begin{repeatamp}
  Let $G$ be a connected reductive group over a field $k$ and $T
  \subseteq G$ a  
  maximal torus.  
  For any
  $G$-linearized ample line bundle on a
  projective $G$-variety $X$ over $k$ satisfying $\stable X T = \semistable X
  T$ and any
  Chow class $\alphanew \in A_0(X \dd G)_\Q$
  with lift $\tilde \alphanew \in A_{\ast}(X\dd T)_\Q$,
  \begin{equation*}
    \int_{X\dd G}\alphanew = \frac{1}{|W|} \int_{X\dd T} \ctop(\mathfrak g / \mathfrak t) \frown
    \tilde \alphanew.
  \end{equation*}
\end{repeatamp}
\begin{proof}
We begin by pointing the reader to
 \cite{sesh1} as a reference on geometric invariant theory relative 
to a base.  The base we will use is the spectrum of the ring $\Z_{(p)}$,
the localization of $\Z$ at 
the prime $p$ equal to the characteristic of the base field $k$.
 
By Theorems \ref{GIT-integral-ratio-is-invariant-of-G-theorem} and 
\ref{GIT-integral-ratio-decomposes-multiplicatively}, it suffices to 
  verify $r_G^X = |W|$ on a single
  projective $G$-variety $X$ for each simple Chevalley group
  $G$.
  Each Chevalley group $G$ admits
  a model $G_{\Z}$ over the integers, with a split maximal torus $T_\Z
  \subseteq G_\Z$.
  We assert that there is a
  smooth projective $\Z_{(p)}$-scheme $X_{(p)}$ on which
  $G_{(p)} := G_\Z \times_\Z {\Z_{(p)}}$ acts as well as a
  $G_{(p)}$-linearized ample line bundle
  for which
  all $G_{(p)}$- (resp.~$T_{(p)}$-) semi-stable points are stable and
  comprise an open locus that nontrivially intersects the closed fibre
  over $\F_p$. 
  We justify this assertion 
  briefly:
  Proposition \ref{trivial-flag-variety-bundle-prop} reduces the problem
  to finding some $\Z_{(p)}$-scheme for which there exist
  $G_{(p)}$-stable points in the closed fibre over $\F_p$; 
  with the aid of the Hilbert-Mumford criterion, one discovers that
  many such schemes exist (e.g. take $\P(V_{\Z_{(p)}}^{\oplus n})$ with
  $V_{\Z_{(p)}}^{\oplus n}$ a large multiple of a general irreducible
  $G_{(p)}$-representation). 
  
  Having chosen such an $X_{(p)}$, the technique of specialization
  (cf.~\cite[\S20.3]{ful1}) implies that the 
  integral of relative
  $0$-cycles on $X_{(p)} \dd G_{(p)}$ 
  and $X_{(p)} \dd T_{(p)}$ restricted to the generic fibre over $\Q$ is
  equal to the integral restricted to any closed fibre over $\F_p$.
  The ratio $r_G$ is  
  independent under field extension by Lemma
  \ref{lemma-independent-under-field-extension}, and so this reduces
  the  calculation of $r_G$ over the field $k$ to the
  computation of $r_{G_{\C}}$, where $G_{\C} := G_{\Z} \times_{\Z}
  \C$.  
  The Kirwan-Kempf-Ness theorem 
  (cf.~\cite[\S8]{kir1} or \cite[\S8.2]{GIT}) shows that in the analytic
  topology, 
  the GIT quotient $X_\C \dd G_\C$ 
  is homeomorphic to the symplectic reduction of $X_{\C}$ by a maximal
  compact subgroup, and so Martin's theorem
  \cite[Thm.~B$'$]{mar1}
  implies $r_{G_\C} = |W|$.
\end{proof}
\begin{question}\label{question-purely-algebraic-proof}
  What is a purely algebraic proof that $r_G = |W|$ for a general
  connected reductive group $G$?
\end{question}
 
In light of Theorems
\ref{GIT-integral-ratio-is-invariant-of-G-theorem} and
\ref{GIT-integral-ratio-decomposes-multiplicatively}, to answer
Question \ref{question-purely-algebraic-proof} it suffices to verify
$r_G = |W|$ for all simple groups $G$.  Such a verification was done
in \S \ref{section-calculation} for simple groups of type $\mathbf
A_{n}$.  Can $r_G$ be calculated (algebraically) for any other simple
groups $G$?

\section*{Appendix: Chow groups and quotient
  stacks}\label{appendix-equivariant-Chow-group} 

\setcounter{equation}{0}
\setcounter{subsection}{0}
\renewcommand{\theequation}{A.\arabic{subsection}.\arabic{equation}}
\renewcommand{\thesubsection}{A.\arabic{subsection}}

Here we recall the basic properties of Chow groups for schemes and
quotient stacks.

\subsection{Chow groups}

For a scheme $X$ defined over a field $k$, let $A_i(X)$ denote the
$\Z$-module generated by $i$-dimensional subvarieties over $k$ modulo rational
equivalence (cf.~\cite{ful1}).  We call $A_\ast(X) := \oplus_i A_i(X)$ the
\emph{Chow group} of $X$.   To indicate rational
coefficients, we write $A_\ast(X)_\Q := A_\ast(X) \otimes \Q$.

For a scheme $X$ over a field $k$ and an algebraic group $G$
acting on $X$, the Chow group of the quotient stack $[X/G]$ is defined
by Edidin and Graham in \cite{edi-gra1} to be the limit of Chow
groups using Totaro's  finite approximation construction (cf.~\cite{tot1}):
$$A_i([X/G]) := A_{i-g+\ell}(X \times_G U),$$
where $U$ is an open subset of an $\ell$-dimensional
$G$-representation $V$ on which the $g$-dimensional group $G$ acts
freely and whose complement 
$V\setminus U$ has sufficiently large codimension.  It is a result
of Edidin and Graham that this group is well-defined, independent of the
presentation of the stack $[X/G]$ as a quotient, and recovers Gillet's
original definition \cite{gil1}
of Chow groups on Deligne-Mumford stacks upon tensoring with $\Q$
(see \cite[\S 5]{edi-gra1}).   
We may also think of $A_i[X/G]$ as the $G$-equivariant Chow group of $X$ to 
highlight the functoriality with respect to group homomorphisms,
and we make use of the notation $A_\ast^G(X) := A_\ast([X/G])$.

One benefit of the stacky interpretation
arises when $X$ admits a geometric invariant theory quotient $X \dd
G$.  In this case, there is a coarse moduli morphism
from the quotient stack $\gitstack X G$ to $X \dd G$ that induces an
isomorphism of Chow groups with rational coefficients:

\begin{theorem}[Edidin-Graham, Gillet, Vistoli]\label{theorem-gillet-vistoli}
If $X$ is a projective $G$-variety with a $G$-linearized ample line
bundle for which $\semistable X G = \stable X G$,
then the induced morphism $\phi: \gitstack X G \to X \dd G$ yields an
isomorphism of rational Chow groups,
$$\phi^\ast:  A_\ast(X \dd G)_\Q \to  A_\ast([X/G])_\Q.$$
\end{theorem}
\begin{proof}
  See \cite[Thm.~4]{edi-gra1} for a proof of the theorem, and the
  remark that follows it for a discussion of the prior results of
  Gillet \cite{gil1} and Vistoli \cite{vis1}.
\end{proof}

The Chow groups of quotient stacks are functorial with respect to the
usual operations (e.g. flat pull-back, proper
push-forward), and when $X$ is a smooth $n$-dimensional variety, there
is an intersection 
product that endows these groups with the structure of a commutative
ring with identity, graded by codimension and denoted by
$A^\ast([X/G]) := A_{n - g - \ast}([X/G])$.
Hence the Chow group of the stack $A_\ast([X/G])$ is
naturally a module over the ring $A^\ast(BG)$ where $BG = [\Spec
  k/G]$.   In the case $T = \G_m^r$ is a split
torus of rank $r$,   
$$A^\ast(BT) \cong \Sym \chargp T \cong \Z[\chi_1,\ldots, \chi_r],$$  
where $\chi_1,\ldots, \chi_r$ is some $\Z$-basis of $\chargp T$.
A character $\chi \in \chargp T$ is equivalent to a line bundle
$\sL_\chi$ over $BT$ whose Chern class $c_1(\sL_\chi) \in A^\ast(BT)$
corresponds to $\chi$ under the above isomorphism.

Not surprisingly, the structure of the Chow group $A_\ast([X/T])$ of the
stacky quotient of $X$ by a torus $T$ is especially 
well-understood (see \cite{bri1}).

\begin{proposition}[Brion]\label{brions-presentation-of-T-equivariant-Chow-proposition}  
  Let $X$ be a scheme with the action of a torus $T$  over an
  algebraically closed field $\bar k$.
  The $T$-equivariant Chow group 
  $A_\ast^T(X)$ is generated as an $A^\ast(BT)$-module by the classes $[Y]$
  associated to $T$-invariant closed subschemes $Y \inject X$.  
\end{proposition}
\begin{proof}
  See \cite[Thm.~2.1]{bri1}.
\end{proof}
\noindent Moreover, there is a
localization theorem useful for making calculations in $T$-equivariant
Chow groups.  The following version of the localization theorem will
suffice for our purposes:

\begin{theorem}[Localization]\label{brions-localization-theorem}
  Let $X$ be a smooth projective scheme with the action of a torus $T$  over an
  algebraically closed field $\bar k$,
  and let 
  $i: X^T \to X$ denote the inclusion of the scheme of $T$-fixed
  points.  Then the morphism
  $$i^\ast: A_\ast^T(X)_\Q \to A_\ast^T(X^T)_\Q$$
  is an injective $A^\ast(BT)$-algebra morphism.
  Furthermore, if $X^T$ consists of finitely many points, then the
  morphism
  $$i^\ast: A_\ast^T(X) \to A_\ast^T(X^T)$$
  of Chow groups with integer coefficients is injective as well.
\end{theorem}
\begin{proof}
  See \cite[Cor. 3.2.1]{bri1}.  
\end{proof}

\subsection{Operational Chow groups}

 The $i$th operational Chow group $A^i(X)$ is defined to be the group of
``operations'' $c$ that comprise a system
of group homomorphisms  $c_f:A_\ast(Y) \to A_{\ast -i}(Y)$ associated to
morphisms of schemes $f:Y \to X$ and compatible 
with proper push-forward, flat pull-back, and the refined Gysin map
(cf.~\cite[\S17]{ful1}). Similarly, Edidin and Graham 
define equivariant operational Chow groups $A^i_G(X)$ via systems of
group homomorphisms $c_{f}^G: A_\ast^G(Y) \to A_{\ast-i}^G(Y)$
compatible with the $G$-equivariant 
analogues of the above maps (cf.~\cite[\S 2.6]{edi-gra1}). 

The clearest examples of equivariant operational Chow classes are 
equivariant Chern classes $c_i(\sE)$ of $G$-linearized vector
bundles $\sE$ (i.e. Chern classes of vector bundles on $[X/G]$).
Moreover, $A_G^\ast(X)$ equipped with composition
forms an associative, graded ring with identity.
 When $X$ is smooth, there is a Poincar\'e 
duality between the equivariant operational Chow group and the
usual equivariant Chow group.  For any operational Chow class $c =
\{c_f^G\} \in
A^i_G(X)$ and Chow class $\sigma \in A_\ast^G(Y)$, we introduce the
following  
``cap product'' notation:
$$c \frown \sigma := c_f^G(\sigma) \in A_{\ast-i}^G(Y).$$

\begin{theorem}[Poincar\'e duality]\label{poincare-duality-theorem}
  If $X$ is a smooth $n$-dimensional variety, then the map $A^i_G(X)
  \to A_{n-i}^G(X)$ defined by $c \mapsto c \frown [X]$ is an
  isomorphism. 
\end{theorem}
\begin{proof}
  See \cite[Prop. 4]{edi-gra1}.
\end{proof}

\begin{remark}
  When $X$ is a smooth $n$-dimensional variety, this allows us to
  write $A^k_G(X)$ to denote the 
  codimension $k$ Chow group $A_{n-k}^G(X)$, without any ambiguity in
  notation.
  Furthermore, this identification forms an
  isomorphism of rings $A_G^\ast(X) \cong
  A_{n-\ast}^G(X)$, with the multiplication structure on
  $A_{n-\ast}^G(X)$ given by the intersection product.
\end{remark}

\end{document}